\documentclass[11pt]{article}
\usepackage{amsmath,amsthm,amsfonts,amssymb,amscd,thmtools}
\usepackage[scr]{rsfso}
\usepackage{stackrel}
 \usepackage{amsmath}
\usepackage{enumerate} 
\usepackage{color}
\usepackage{float}
\usepackage{soul}
\usepackage{graphicx}
\usepackage[colorlinks=true]{hyperref}
\usepackage{mathpazo}
\usepackage{fancyhdr}
\usepackage{empheq}
\usepackage[margin=1in]{geometry}

\def\XX{\mathbb{X}}
\def\YY{\mathbb{Y}}

\def\Hat{\widehat}

\def\Bar{\overline}
\def\ra{\rangle}
\def\la{\langle}
\def\ve{\varepsilon}
\def\B{\mathbb{B}}
\def\h{\hfill\Box}
\def\R{\mathbb{R}}
\def\ox{\bar{x}}
\def\OX{\Bar{X}}
\def\oy{\bar{y}}
\def\OY{\Bar{Y}}

\def\ov{\bar{v}}
\def\OV{\Bar{V}}

\def\OW{\Bar{W}}

\def\OU{\Bar{U}}

\def\cone{\mbox{\rm cone}\,}

\def\Im{\mbox{\rm Im}\,}

\def\rank{\mbox{\rm rank}\,}
\def\Tr{\mbox{\rm Tr}\,}

\def\dom{\mbox{\rm dom}\,}
\def\Ker{\mbox{\rm Ker}\,}

\def\bd{\mbox{\rm bd}\,}

\def\cl*co{\mbox{\rm cl}^*\mbox{\rm co}\,}
\newcommand{\eqdef}{\stackrel{\text{\rm\tiny def}}{=}}
\def\cl{\mbox{\rm cl}\,}

\def\h{\hfill\triangle}
\def\dn{\downarrow}

\def\TT{\mathbb{T}}

\def\oR{\Bar{\R}}

\def\lm{\lambda}

\def\al{\alpha}

\def\Lm{\Lambda}

\def\hs7{\hspace*{7pt}}

\def\Id{\mathbb{I}}

\renewcommand{\theequation}{\thesection.\arabic{equation}}

\def\h{\hfill\Box}
\def\kk{\kappa}
\begin{document}

\newtheorem{Theorem}{Theorem}[section]
\newtheorem{Conjecture}[Theorem]{Conjecture}
\newtheorem{Proposition}[Theorem]{Proposition}
\newtheorem{Remark}[Theorem]{Remark}
\newtheorem{Lemma}[Theorem]{Lemma}
\newtheorem{Corollary}[Theorem]{Corollary}
\newtheorem{Definition}[Theorem]{Definition}
\newtheorem{Example}[Theorem]{Example}
\newtheorem{Fact}[Theorem]{Fact}
\newtheorem*{pf}{Proof}
\renewcommand{\theequation}{\thesection.\arabic{equation}}
\normalsize
\normalfont
\medskip
\def\endproof{$\h$\vspace*{0.1in}}

\title{\bf Solution uniqueness of convex optimization problems via the radial cone}
\date{}
\author{Jalal Fadili\footnote{Normandie Universit\'e, ENSICAEN, UNICAEN, CNRS, GREYC, France; email: jalal.fadili@ensicaen.fr}\,, \;\; Tran T. A. Nghia\footnote{Department of Mathematics and Statistics, Oakland University, Rochester, MI 48309, USA; email: nttran@oakland.edu}\,,\; \; and \;\; Duy Nhat Phan\footnote{University of Dayton Research Institute, University of Dayton, Dayton, OH 45469, USA; duynhat.phan@udri.udayton.edu} }

\maketitle
{\small \noindent {\bf Abstract.} In this paper, we mainly study solution uniqueness of some convex optimization problems. Our characterizations of solution uniqueness are in terms of the radial cone. This approach allows us to know when a unique solution is a strong solution or even a tilt-stable one without checking second-order information. Consequently, we apply our theory to low-rank optimization problems. The radial cone is fully calculated in this case and numerical experiments show that our characterizations are sharp.

}
\section{Introduction}
In many disciplines of science and engineering, solving the following linear inverse problem is essential:
\begin{equation}\label{IP}
    \Phi x=y_0,
\end{equation}
where $\Phi:\XX\to \YY$ is a linear operator between two finite dimensional vector spaces and $y_0=\Phi x_0$ is an element of $\YY$. Particularly, in signal/image processing, $y_0$ is a known observation of an unknown signal $x_0\in \XX$ through the forward linear operator $\Phi$. In regression in statistical learning, $y_0$ is the response vector, $x_0\in \XX$ is the unknown parameter vector, and $\Phi$ is the design matrix. Recovering exactly the signal $x_0$ by solving the linear system \eqref{IP} through least-squares raises non-uniqueness issues, as usually the dimension of $\YY$ is much smaller than the dimension of $\XX$, which makes \eqref{IP} underdetermined and thus has infinitely many solutions. On the other hand, the vector $x_0$ usually possesses some special features (in the form of some low complexity) that cannot be captured by naive least-squares. Thus, a form complexity regularization is necessary in which the recovery of $x_0$ is cast as the regularized optimization problem:
\begin{equation}\label{P0}
    \min_{x\in \XX}\quad g(x)\quad \mbox{subject to}\quad \Phi x=y_0,
\end{equation}
where $g:\XX\to \R$ is some particular function/regularizer, which is chosen to promote the low complexity, i.e. the expected features, of the original vector $x_0$. Solving this optimization problems narrows down the possible candidate solutions for $x_0$. And if we want the solution of \eqref{P0} to be exactly $x_0$, it should hopefully be the unique solution. That is one of the main reasons why solution uniqueness of problem \eqref{P0} has drawn much attention from different authors and research groups; see, e.g., \cite{BDE09,CRPW12,CR09,CT05,DH01,FR13,F05,FNT21,G17,GHS11,HP23,MS19,RF08,T13,VPDF13,ZYC15} for an incomplete list. Although the regularizer $g$ can be nonconvex in many problems \cite{BDE09,FR13}, in this paper we focus on the case when it is a continuous convex function. Unlike system \eqref{IP}, problem \eqref{P0} likely has a unique solution provided that the number of observations, i.e., the dimension of $\YY$ is large enough compared to some intrinsic dimension of $x_0$ (but still much smaller than that of $\XX$) \cite{CRPW12,CR09,CR13}.

Another optimization problem relative to \eqref{P0} studied in our paper is 
\begin{equation}\label{CP0}
    \min_{x\in \XX}\quad f(\Phi x)+g(x),
\end{equation}
where $f:\YY\to \oR$ is a proper lower semicontinuous (l.s.c.) convex function that is twice differentiable with positive definite Hessian in the interior of its domain. This problem is often used when there is noise in the observation $y_0$ and the loss function $f$ measures the error between $\Phi x$ and the noisy observation. An important class of problems \eqref{P0} and \eqref{CP0} in signal processing and machine learning is when $g$ is the $\ell_1$-norm in $\XX=\R^n$, which is employed to promote  sparsity features of  their optimal solutions. Solution uniqueness of $\ell_1$-optimization is characterized by the so-called  {\em Nondegenerate Source Condition} and {\em Restricted Injectivity} of operator $\Phi$ \cite{CT05,F05,GHS11,VPDF13} or the {\em Null Space Property} and its
variants \cite{DH01,FR13}; see also \cite{T13,ZYC15,ZYY16,G17,BLN22} for different approaches and characterizations. 

When $g$ is the $\ell_1/\ell_2$ norm in $\XX=\R^n$ (also known as the {\em group Lasso regularizer}), a popular regularizer in image processing and statistics when the sparsity is structured in groups, sufficient conditions for solution uniqueness of problem \eqref{P0} were obtained in \cite{G11,RRN12,JKL15,RF08}. Full characterizations of this property for $\ell_1/\ell_2$ optimization problems were recently established in \cite{FNT21} via {\em second-order variational analysis} and the exact computation of the {\em descent cone} of $\ell_1/\ell_2$ norm introduced in \cite{CRPW12}. Another important class of functions $g$ popular in matrix optimization is the {\em nuclear norm} $\|\cdot\|_*$ in the space of matrices $\XX=\R^{n_1\times n_2}$. Problem \eqref{P0} in this case is known as the {\em nuclear norm minimization problem}:
\begin{equation}\label{NNM0}
\min_{X\in \R^{n_1\times n_2}}\quad \|X\|_*\quad \mbox{subject to}\quad \Phi X=Y_0,
\end{equation}
where $\Phi:\R^{n_1\times n_2}\to \R^m$ is a linear operator and $Y_0$ is a vector in $\R^m$. Employing  the nuclear norm enforces the low-rank property of optimal solutions that is required in many matrix optimization problems; see, e.g.,  \cite{CP10,CR09,CR13}. Sufficient conditions for solution uniqueness of problem~\eqref{NNM0} play important roles in obtaining a small bound on the number of measurements $m$ such that solving problem~\eqref{NNM0} recovers exactly the solution $X_0$ \cite{CRPW12,CR09,CR13}. Solution uniqueness of this problem can be characterized via the descent cone in \cite{CRPW12}, whose calculation is not known yet. Recently, \cite{HP23} obtained a new characterization for this property. Their condition is verifiable in low dimensional spaces, but looks very complicated in higher dimensions. A different characterization is derived in \cite[Lemma~3.1]{LPB21} via the so-called {\em radial cone} \cite{BS00} of the nuclear norm; see also \cite[Theorem~3.1]{HKS23} for a more general class of {\em gauge functions}. However, this radial cone is not fully calculated there and, according to the authors of \cite{LPB21}, it is not easy to check the corresponding characterization (possibly due to non-closedness of the radial cone). Alternatively, those authors use the closure of the radial cone, which is known as the {\em tangent cone} to check solution uniqueness. In the recent paper \cite{FNP23}, it is shown that using the tangent cone instead of the radial cone can characterize  {\em strong minima} \cite{BS00} of problem~\eqref{NNM0}.

Among three different ways of characterizing solution uniqueness of problem~\eqref{NNM0} in \cite{CRPW12,LPB21,HP23} mentioned above, we argue that using the radial cone is the simplest one. In this paper, we push this direction forward for general problems \eqref{P0} and \eqref{CP0} by providing more analysis, exact calculations of the radial cone, numerical methods, and making interconnections with some other advanced results for characterizing {\em sharp minima} \cite{C78,P79,FNT21}, {\em strong minima} \cite{BS00,FNP23}, as well as {\em tilt stability} \cite{PR98,MR12,LZ13}. 

\medskip

\paragraph{Contributions} In Theorem~\ref{UniRa}, we show that $\ox$ is the unique optimal solution of problem \eqref{CP0} if and only if $\oy=-\Phi^*\nabla f(\Phi \ox)\in \partial g(\ox)$ and 
\begin{equation}\label{R1}
    \Ker \Phi \cap \mathcal{R}_{(\partial g)^{-1}(\oy)}(\ox)=\{0\}, 
\end{equation}
 where $\mathcal{R}_{(\partial g)^{-1}(\oy)}$ is the radial cone at $\ox$ to the set $(\partial g)^{-1}(\oy)$ defined in Section~\ref{sec:prelim}. For problem \eqref{P0}, we also have a similar result that $x_0$ is the unique solution if and only if there exists a {\em dual certificate} $y\in  \partial g(x_0)\cap \Im \Phi^*$ such that 
\begin{equation}\label{R2}
    \Ker \Phi \cap \mathcal{R}_{(\partial g)^{-1}(y)}(x_0)=\{0\}.  
\end{equation}
This characterization was constructed in \cite[Lemma~3.1]{LPB21} for the case of nuclear norm minimization \eqref{NNM0} and recently in \cite[Theorem~3.1]{HKS23} for convex gauge functions with some additional assumptions. Our results work for any continuous convex  functions $g$. Most importantly, we advance this approach in the following ways: 
\begin{itemize}
\item We establish the connection between  solution uniqueness and strong minima \cite{BS00,RW98} for problems \eqref{P0} and \eqref{CP0}. Obviously, a strong solution is a unique solution for convex optimization problems. The opposite implication is not true in general. But when $g$ is the $\ell_1/\ell_2$ norm,  \cite{FNT21} recently shows that unique solutions of problems \eqref{P0} and \eqref{CP0} are also strong ones via second-order analysis. It is not well-understood the reason why the structure of $\ell_1/\ell_2$ norm can give such an equivalence.  In this paper, we have a clearer answer for this phenomenon. Closedness of the radial cone plays an important role there.  Indeed, we prove this equivalence in Corollary~\ref{StrUnq} in the case when  the function $g$ satisfies the {\em quadratic growth condition} \cite{BS00,FNP23} and the radial cone in \eqref{R1} and \eqref{R2} is closed; see also Remark~\ref{StroUniq} for similar analysis on problem~\eqref{P0}. This is quite a surprise because usually one has to check second-order information for strong minima, while our conditions \eqref{R1} and \eqref{R2} only use first-order information.

\item We unify the two approaches of using radial cones and descent cones for solution uniqueness. As discussed above, solution uniqueness is also characterized by descent cones introduced in \cite{CRPW12}. In the recent paper \cite[Theorem~5.1]{FNT21}, the descent cone of $\ell_1/\ell_2$ norm (and thus $\ell_1$ norm) has an exact computation via its first-order information. We generalize that result in this paper for any continuous convex function $g$. Exact formulae of the descent cone via the radial cone and the {\em critical cone} are provided in Proposition~\ref{DesRd} and Proposition~\ref{Int}. These calculations allow us to link up the two aforementioned approaches for solution uniqueness in Theorem~\ref{UniRd2}. Consequently, the descent cone of the nuclear norm is also computed exactly in this paper. 

\item We are able to apply our general theory to the nuclear norm minimization problem \eqref{NNM0}. Although solution uniqueness of  problem \eqref{P0} was already characterized in \cite[Lemma~3.1]{LPB21} via the radial cone, it was not fully calculated there. In Lemma~\ref{RC}, we provide a complete computation of this radial cone via  initial data. This leads us to a full characterization of solution uniqueness for problem~\eqref{NNM0} in Corollary~\ref{UniRaNu}.  Other interesting connections of solution uniqueness with sharp minima \cite{C78,P79,CR13,FNT21} and tilt stability \cite{PR98,MR12,LZ13,DMN14} for nuclear norm minimization problems are also obtained in Sections~\ref{sec:uniquecomp} and \ref{sec:uniquelin}. One of our most exciting results is to show the equivalence between solution uniqueness and tilt stability for problem \eqref{CP0} in Corollary~\ref{Tilt} under {\em Nondegenerate Condition} (a.k.a. {\em Strict Complementary Condition} \cite{BS00}) when $g$ is the nuclear norm. This implicitly  indicates that solution uniqueness makes problem \eqref{CP0} stable to be solved. Numerical experiments on the nuclear norm minimization problem \eqref{NNM} in Section~\ref{NE} confirm that our characterizations of solution uniqueness are sharp and computationally friendly.
\end{itemize}

\section{Preliminaries}\label{sec:prelim}
\setcounter{equation}{0} 

Throughout this paper, let $\XX$ be an Euclidean space and $\XX^*$ be its dual space endowed with the inner product $\la v,u\ra$ for $v\in \XX^*$ and $u\in \XX$. Suppose that $\|\cdot\|$ is the  norm in $\XX$ induced by the latter inner product. We denote $\B_r(\ox)$ by the open ball with center $\ox\in \XX$ and radius $r>0$. We first recall some notions and classical results from convex analysis that will be important to our exposition; see e.g. \cite{R70} for a comprehensive account.

Let $\Omega$ be a closed convex set in $\XX$ and $\ox\in \Omega$. We write ${\rm cone}\, \Omega$ for the {\em conic hull} of $\Omega$ and ${\rm aff }\, \Omega$ for the {\em affine hull} of $\Omega$.
The {\em relative interior} of $\Omega$ is denoted by
\[
{\rm ri}\, \Omega=\{x\in \Omega|\;\exists \ve>0, \B_\ve(x)\cap  {\rm aff }\, \Omega\neq \emptyset\}; 
\]
 The {\em radial cone} (known also the {\em cone of feasible directions}) and the {\em tangent cone} to $\Omega$ at $\ox$ are defined, respectively, by
\begin{equation}\label{RT}
    \mathcal{R}_\Omega(\ox)={\rm cone}\,(\Omega-\ox)=\R_+(\Omega-\ox)\quad \mbox{and}\quad \mathcal{T}_\Omega(\ox)=\overline{\mathcal{R}_\Omega(\ox)},
\end{equation}
which means that $\mathcal{T}_\Omega(\ox)$ is  the topological closure of the radial cone $\mathcal{R}_\Omega(\ox)$; see, e.g., \cite[Definition~2.54 and Proposition~2.55]{BS00}. When $\Omega$ is a {\em polyhedral} set, i.e., it can be represented as the intersection of finitely many closed half spaces, the radial cone is closed and thus is exactly the tangent cone at the same point $\ox\in \Omega$. But in general, the radial cone is not closed and strictly smaller than the tangent cone. 

Now let $\varphi:\XX\to \oR\eqdef\R\cup\{\infty\}$ be a proper lower semi-continuous (l.s.c.) convex function with  nonempty domain $\dom \varphi\eqdef\{x\in \XX|\; \varphi(x)<\infty\}.$ For $\ox\in \dom \varphi$, the {\em subdifferential} of $\varphi$ at $\ox$ is defined by
\[
\partial\varphi(\ox)\eqdef\{v\in \XX^*|\; \varphi(x)-\varphi(\ox)\ge \la v,x-\ox\ra, \forall x\in \XX\}.
\]
The {\em Legendre-Fenchel conjugate} of $\varphi$ is the l.s.c. convex function $\varphi^*:\XX^*\to \oR$ defined by
\[
\varphi^*(v)\eqdef\sup\{\la v,x\ra-\varphi(x)|\; x\in \XX\}\quad \mbox{for}\quad v\in \XX^*.
\]
Note that $v\in \partial \varphi(x)$ if and only if $x\in \partial \varphi^*(v)$, i.e., $\partial \varphi^*(v)=(\partial \varphi)^{-1}(v) = \{x\in X|\; v\in \partial \varphi(x)\}$, which plays an important role in our study later in the paper. The directional derivative of $\varphi$ at $\ox$ in some direction $w\in \XX$ is
\begin{equation*}
    d\varphi (\ox)(w)\eqdef\lim_{t\dn 0}\dfrac{\varphi(\ox+tw)-\varphi(\ox)}{t}. 
\end{equation*}
If $\varphi$ is continuous around $\ox$, it is well-known that 
\begin{equation}\label{Dd}
    d\varphi (\ox)(w)= \sup_{v\in \partial \varphi(\ox)}\la v,w\ra. 
\end{equation}

The point $\ox$ is called a minimizer or an optimal solution of the following problem
\begin{equation}\label{OP}
\min_{x\in \XX}\quad \varphi(x),
\end{equation}
if $\varphi(x)\ge \varphi(\ox)$ for all $x\in \XX$, i.e., $0\in \partial \varphi(\ox)$ or $\ox\in \partial \varphi^*(0)$. We say $\ox$ is a {\em unique} (or {\em strict}) solution  problem~\eqref{OP} if  $\varphi(x)> \varphi(\ox)$ for all $x\in \XX\setminus\{\ox\}$, which means that  $\{\ox\}= \partial \varphi^*(0)$. Moreover, $\ox$ is said to be a {\em sharp solution} if there exists some modulus $c>0$ such that 
\[
\varphi(x)-\varphi(\ox)\ge c\|x-\ox\|\quad \mbox{for all}\quad x\in \XX. 
\]
The notion of sharp minima was introduced independently by Polyak \cite{P79} and Crome \cite{C78} under the name {\em strong uniqueness}. On the other hand, $\ox$ is a {\em strong solution} if there exist $\ve>0$ and modulus $\kk>0$ such that 
\[
\varphi(x)-\varphi(\ox)\ge \frac{\kk}{2}\|x-\ox\|^2\quad \mbox{for all}\quad x\in \B_\ve(\ox).
\]
The notion of strong minima is weaker than that of sharp minima. For instance, the former is only a local property. We refer the readers to \cite{BS00,P87, RW98} for different ways to characterize sharp/strong minima and their roles in optimization. 

\medskip

Finally in this section, we recall the {\em quadratic growth condition} for functions. 
\begin{Definition}[Quadratic growth condition]\label{def:QGC} The convex function $\varphi$ is said to satisfy the quadratic growth condition at $\ox\in \dom \varphi$ for $\ov\in \partial \varphi(\ox)$ if there exist $\ve>0$ and modulus $\kk>0$ such that 
\begin{equation}\label{QGC}
\varphi(x)-\varphi(\ox)-\la \ov,x-\ox\ra\ge\frac{\kk}{2}{\rm dist}\,(x;\partial \varphi^*(\ov))^2     \quad \mbox{for all}\quad x\in \B_\ve(\ox).
\end{equation}
\end{Definition}
When the function $\varphi$ is convex, the quadratic growth condition is equivalent to the \emph{\L{}ojasiewicz inequality with exponent $\frac{1}{2}$} \cite{BNPS17} and the  \emph{metric subregularity of the subdifferential}   \cite{AG08, DMN14, ZT95}.  There are  broad classes of convex functions satisfying the quadratic growth condition such as \emph{piece-wise linear quadratic convex} functions \cite[Definition 10.20]{RW98} and many {\em convex spectral} functions \cite{CDZ17}; see also \cite{ZS17} for some other ones. 

\section{Solution uniqueness of convex composite optimization problems}\label{sec:uniquecomp}
\setcounter{equation}{0} 

\subsection{General regularizer}
In this section we study the solution uniqueness of the following convex composite optimization problem
\begin{equation}\label{CP}
\min_{x\in \XX}\quad \varphi(x)\eqdef f(\Phi x)+g(x),
\end{equation}
where $\Phi:\XX\to\YY$ is a linear operator between two Euclidean spaces, $f:\YY\to \oR$ is a proper l.s.c. convex function that is differentiable in the interior of its domain, and $g:\XX\to \oR$ is a proper l.s.c. convex function with 
\begin{equation}\label{SA}
\Phi^{-1}\left({\rm int}\, (\dom f)\right)\cap \dom g\neq \emptyset.
\end{equation}
Throughout this section, we have two additional standing assumptions on the loss function $f$:
\begin{enumerate}[ \rm (a) ]
\item $f$ is twice continuously  differentiable in ${\rm int}\, (\dom f)$. \label{aA}
\item The Hessian $\nabla^2 f(y)$ is positive definite for all $y\in {\rm int}\, (\dom f)$, which implies that $f$ is strictly convex in the interior of its domain. \label{aB}
\end{enumerate}
The class of loss functions satisfying the above standing assumptions includes  strongly convex and twice continuously differentiable functions \cite{ZS17} together with the discrete {\em Kullback-Leiber divergence} (known also as {\em relative entropy}) \cite{C91} that have been used widely in statistical/machine learning and signal processing.

Our first result establishes the characterization for solution uniqueness of problem~\eqref{CP} via the radial cone. 

\begin{Theorem}[Solution uniqueness via the radial cone]\label{UniRa} Let $\ox\in \Phi^{-1}\left({\rm int}\, (\dom f)\right)\cap \dom g$ be an optimal solution of problem \eqref{CP} and $\oy\eqdef-\Phi^*\nabla f(\Phi\ox)\in \partial g(\ox)$. Then $\ox$ is the unique solution if and only if 
\begin{equation}\label{KR}
    \Ker \Phi\cap \mathcal{R}_{\partial g^*(\oy)}(\ox)=\{0\}. 
\end{equation}
\end{Theorem}
\begin{proof} Suppose that $\ox$ is a unique solution of problem \eqref{CP}. Pick any $w\in \Ker \Phi\cap \mathcal{R}_{\partial g^*(\oy)}(\ox)$, we find some  $t>0$ such that $\ox+tw\in \partial g^*(\oy)$. As $\partial g^*(\oy)$ is a convex set and $\ox\in \partial g^*(\oy)$, we have $\ox+\alpha tw\in \partial g^*(\oy)$ for any $\al\in (0,1)$. Since $\Phi^{-1}\left({\rm int}\, (\dom f)\right)$ is an open set containing $\ox$, $\ox+\alpha tw\in \Phi^{-1}({\rm int}\, (\dom f))$ for sufficiently small $\al>0$. It follows that 
\[
\partial \varphi(\ox+\al tw)=\Phi^*\nabla f(\Phi(\ox+\al tw))+\partial g(\ox+\al tw)=\Phi^*\nabla f(\Phi\ox)+\partial g(\ox+\al tw)\ni 0. 
\]
Hence $\ox+\al tw$ is an optimal solution of problem \eqref{CP}. Since $\ox$ is the unique solution of \eqref{CP}, we have $\ox=\ox+\al tw$, i.e., $w=0$. This implies condition \eqref{KR}.

Conversely, suppose that condition~\eqref{KR} holds. Let $\tilde{x}$ be any optimal solution of problem~\eqref{CP}. It follows that the whole segment $[\ox,\tilde{x}]$ belong to the solution set of \eqref{CP}. Setting $w=\tilde{x}-\ox$, there exists $t_0>0$ sufficiently small such that $\ox+t_0w\in \Phi^{-1}\left({\rm int}\, (\dom f)\right)$. Moreover,
\begin{equation}\label{Inl}
-\Phi^*\nabla f( \Phi \ox)\in \partial g(\ox)\quad  {\rm and}\quad  -\Phi^*\nabla f( \Phi (\ox+t_0w))\in \partial g(\ox+t_0w).
\end{equation}
Due to the monotonicity of $\partial g$, we have
\begin{equation}\label{Mn}
0\le \la -\Phi^*\nabla f( \Phi (\ox+t_0w))+\Phi^*\nabla f( \Phi \ox),t_0w\ra= t_0\la -\nabla f( \Phi (\ox+t_0w))+\nabla f( \Phi \ox),\Phi w\ra.
\end{equation}
As $f$ is twice continuously differentiable in ${\rm int}\, (\dom f)$, we obtain from the mean-value theorem that 
\begin{equation}\label{nabla}
\la\nabla f( \Phi (\ox+t_0w))-\nabla f( \Phi \ox),\Phi w\ra=t_0\int_0^1\la \nabla^2 f(\Phi(\ox+tt_0w))\Phi w,\Phi w\ra dt.
\end{equation}
Since the Hessian mapping $\nabla^2 f$ is positive definite and continuous on ${\rm int}\, (\dom f)$,  there exists some $\mu>0$ such that 
\[
\la\nabla^2 f(\Phi(\ox+tt_0w))\Phi w,\Phi w\ra \ge \mu \|\Phi w\|^2
\]
for all $t\in [0,1]$. Combining this with \eqref{nabla} and \eqref{Mn} gives us that $\Phi w=0$, i.e., $w\in \Ker \Phi.$ Observe from the inclusion \eqref{Inl} that 
\[
\oy=-\Phi^*\nabla f( \Phi \ox)=-\Phi^*\nabla f( \Phi (\ox+t_0w))\in \partial g(\ox+t_0w),
\]
which implies that $\ox+t_0w\in \partial g^*(\oy)$, i.e., $w\in \mathcal{R}_{\partial g^*(\oy)}(\ox)$. As condition~\eqref{KR} is satisfied, we have $w=0$. Thus $\tilde{x}=\ox$ which entails uniqueness of $\ox$.
\end{proof}

When the function $g$ additionally satisfies the quadratic growth condition in Definition~\ref{def:QGC} at $\ox$ for $\oy=-\Phi^*\nabla f(\Phi\ox)$, the recent paper \cite[Corollary~3.4]{FNP23} shows that the following condition
\begin{equation}\label{Strong}
\Ker \Phi\cap \mathcal{T}_{\partial g^*(\oy)}(\ox)=\{0\}
\end{equation}
characterizes strong minima at the optimal solution $\ox$, where $\mathcal{T}_{\partial g^*(\oy)}(\ox)$ is the tangent cone \eqref{RT} to $\partial g^*(\oy)$ at $\ox$. As the tangent cone is the closure of the radical cone. Condition~\eqref{Strong} coincides with condition~\eqref{KR} when the radial cone $\mathcal{R}_{\partial g^*(\oy)}(\ox)$ is closed.  This observation allows us to establish the equivalence between solution uniqueness and strong minima for problem~\eqref{CP} in the following result. 

\begin{Corollary}[Strong minima and solution uniqueness]\label{StrUnq} Suppose that $\ox\in \Phi^{-1}\left({\rm int}\, (\dom f)\right)\cap \dom g$ is an optimal solution of problem~\eqref{CP} and the function $g$ satisfies the quadratic growth condition at $\ox$ for $\oy=-\Phi^*\nabla f(\Phi\ox)$. Suppose further that the radial cone $\mathcal{R}_{\partial g^*(\oy)}(\ox)$ is closed. Then $\ox$ is a strong solution of problem~\eqref{CP} if and only if it is a unique solution. 
\end{Corollary}


\begin{Remark}[Closedness of the radial cone]\label{Cl12}{\rm The radial cone $\mathcal{R}_{\partial g^*(\oy)}(\ox)$ is closed if and only if  
\begin{equation}\label{RdCl}
  \cone\left(\partial g^*(\oy)-\ox\right) \mbox{ is closed}.  
\end{equation}
This condition is automatic when $\partial g^*(\oy)$ is a polyhedral set. A particular example of functions 
satisfying both \eqref{RdCl} and the quadratic growth condition in Corollary~\ref{StrUnq} is the class of {\em convex piecewise linear-quadratic functions} \cite[Definition~10.20]{RW98}, i.e., $\dom g$ is a union of finitely many polyhedral sets, relative to each of which the function $g(x)$ has a convex quadratic expression; see also \cite[Proposition~10.21 and Theorem~11.24]{RW98}. This class of functions covers some important regularizers such as  $\ell_1$ (sparsity) norm, the {\em anisotropic total variation}, and the {\em elastic net regularizer} widely used in optimization, statistical/machine learning, and signal processing. The equivalence between solution uniqueness and strong minima for this class was observed in \cite{BLN22}. That allows \cite{BLN22} to obtain some new characterizations of solution uniqueness for $\ell_1$ optimization problems (Lasso in particular) by using second-order analysis for strong minima, different from those in \cite{F05,T13,ZYC15,ZYY16,G17}. 

Another important regularizer that is not convex piecewise linear-quadratic is the $\ell_1/\ell_2$  norm (a.k.a. group Lasso regularizer) defined by
\begin{equation}\label{g12}
g(x)=\|x\|_{1,2}\eqdef\sum_{J\in \mathcal{J}} \|x_J\|\quad \mbox{for}\quad x\in \XX=\R^n,
\end{equation}
where $\mathcal{J}$ is a partition of $\{1,2,\ldots,n\}$ with $p$ distinct groups. It is shown in \cite{ZS17} that $\partial g^*(\oy)$ is a polyhedral set. Thus the radial cone $\mathcal{R}_{\partial g^*(\oy)}(\ox)$ is also closed. It is also known that the $\ell_1/\ell_2$ norm satisfies the quadratic growth condition at $\ox$ for $\oy$; see also \cite{ZS17}. By Corollary~\ref{StrUnq} again, $\ox$ is a strong solution of group Lasso problem~\eqref{CP} if and only if it is a unique solution. This observation was obtained recently in \cite[Theorem~5.3]{FNT21} by a different approach via second-order analysis and the computation of the descent cone of the $\ell_1/\ell_2$ norm. 

Another situation where condition \eqref{RdCl} holds is when the {\em dual nondegeneracy condition} is satisfied in the sense that 
\begin{equation}\label{DuNon}
\ox\in {\rm ri}\, \partial g^*(\oy).
\end{equation}
Indeed, this condition means that $\cone\left(\partial g^*(\oy)-\ox\right)$ is the subspace in $\XX$ parallel to $\partial g^*(\oy)$, which is closed in finite dimension. \hfill$\triangle$
}    
\end{Remark}

\subsection{Nuclear norm case}
We now study a major class of problem~\eqref{CP}, the nuclear norm optimization problem for low-rank matrix recovery: 
\begin{equation}\label{NNM1}
    \min_{X\in \R^{n_1\times n_2}}\quad f(\Phi X)+\|X\|_*, 
\end{equation}
where $\Phi:\R^{n_1\times n_2}\to \R^m$ is a linear operator ($n_2\ge n_1$), the function $f:\R^m\to \oR$ is a proper l.s.c. convex function satisfying the standing assumptions \eqref{aA}-\eqref{aB}, and $g(X)=\|X\|_*$ is the nuclear norm of $X\in \R^{n_1\times n_2}$, which is the sum of all singular values of $X$. 
The dual norm of nuclear norm is known as the {\em spectral norm} $\|X\|$, the largest singular value of $X$. 

We work under the Forbenius inner product in $\R^{n_1\times n_2}$
\[
\la X,Y\ra={\rm Tr}\, (X^TY)\quad \mbox{for all}\quad X,Y\in \R^{n_1\times n_2},
\]
where ${\rm Tr}\,(\cdot)$ is the trace operator in $\R^{n_1\times n_2}$. 
This induces the Frobenius norm in $\R^{n_1\times n_2}$ as follows
\[
\|X\|_F\eqdef\sqrt{\la X,X\ra}=\sqrt{{\rm Tr}\, (X^TX)}\quad \mbox{for all}\quad X\in \R^{n_1\times n_2}.
\]
We recall that for $\OX\in \R^{n_1\times n_2}$, the singular valued decomposition (SVD, in brief) of $\OX$ reads
\begin{equation}\label{OX}
\Bar X=U\begin{pmatrix}\Bar \Sigma_r &0\\
0& 0\end{pmatrix}_{n_1\times n_2}V^T \quad \mbox{with}\quad
\Bar\Sigma_r=\begin{pmatrix}\sigma_1(\OX)& \dots &0\\
\vdots&\ddots &\vdots\\
0 &\ldots & \sigma_r(\OX)\end{pmatrix},
\end{equation}
where $r={\rm rank}\, (\Bar X)$, $U\in \R^{n_1\times n_1}$ and $V\in \R^{n_2\times n_2}$ are orthogonal matrices, and $\sigma_i(\OX)$ are the ordered positive singular values of $\OX$, i.e. $\sigma_1(\OX) \ge\sigma_2(\OX)\ge \ldots\ge \sigma_r(\OX)>0$.  Denote $\mathcal{O}(\OX)$ by the set of all such pairs $(U,V)$ satisfying \eqref{OX}. We write $U=\begin{pmatrix} U_I&U_J\end{pmatrix}$ and $V=\begin{pmatrix} V_I&V_K\end{pmatrix}$, where  $U_I$ and $V_I$ are the submatrices of the first $r$ columns of  $U$ and $V$, respectively. It follows  from \eqref{OX} that $\OX=U_I\Bar\Sigma_rV_I^T$, which is known as a {\em compact or reduced SVD} of $\OX$.

The subdifferential of the nuclear norm was computed in  \cite[Example~2]{W92}. The subdifferential of its Legendre-Fenchel conjugate (i.e. the normal cone to the spectral ball) was formalized in \cite[Proposition 10]{ZS17}, which can be obtained from \cite[Example~1]{W92} via classical convex analysis calculus, i.e. the formula of the normal cone to a level set \cite[Corollary 23.7.1]{R70}. We summarize these results in the following lemma which plays an important role in our paper. 

\begin{Lemma}[Subdifferentials of the nuclear norm and its conjugate]\label{Lem} The subdifferential of the  nuclear norm at $\OX\in \R^{n_1\times n_2}$ is computed by
\begin{equation}\label{subdif}
    \partial\|\OX\|_*=\left\{U\begin{pmatrix} \Id_r&0\\ 0 & W\end{pmatrix}V^T|\; \|W\|\le 1\right\} \quad \mbox{for any}\quad (U,V)\in \mathcal{O}(\OX). 
\end{equation}
Moreover, $\OY\in \partial\|\OX\|_*$ if and only if $\|\OY\|\le 1$ and 
\begin{equation}\label{Fen}
    \|\OX\|_*=\la \OY,\OX\ra. 
\end{equation}
Furthermore, for any $\OY\in\B\eqdef\{Z\in \R^{n_1\times n_2}|\;\|Z\|\le 1\}$, we have 
\begin{equation}\label{Inver}
\partial g^*(\OY)=U\begin{pmatrix}\mathbb{S}_+^{p(\OY)} &0\\0 &0\end{pmatrix}V^T\quad \mbox{for any}\quad (U,V)\in \mathcal{O}(\OY),
\end{equation}
where $\mathbb{S}_+^p$ is the set of all $p\times p$  symmetric positive semidefinite matrices and  $p(\OY)$ is defined by
\begin{equation}\label{p}
p(\OY)\eqdef\#\{i|\; \sigma_i(\OY)=1\}.
\end{equation}
\end{Lemma}

 Let $\Bar Y\in \partial \|\Bar X\|_*$ and $(U,V)\in \mathcal{O}(\OX)$. Note  from \eqref{subdif} that $\OY$  can be written as
\begin{equation}\label{OY}
\Bar Y=U\begin{pmatrix}\Id_r &0\\
0& \Bar W\end{pmatrix}V^T
\end{equation}
with some $\Bar W\in \R^{(n_1-r)\times (n_2-r)}$ satisfying $\|\Bar W\|\le 1$. 
Let $(\Hat U,\Hat V)\in \mathcal{O}(\Bar W)$ and $\Hat U\Sigma \Hat V^T$ be a full SVD of $\OW$.  We get from \eqref{OY} that 
\begin{equation}\label{OY2}
\Bar Y=\OU\begin{pmatrix}\Id_r &0\\
0& \Sigma\end{pmatrix}\OV^T\quad \mbox{with}\quad \OU\eqdef(U_I\; U_J\Hat U)\quad \mbox{and}\quad \OV\eqdef(V_I \; V_K\Hat V).
\end{equation}
Observe that  $\OU^T\OU=\Id_{{n_1}}$ and $\OV^T\OV=\Id_{n_2}$. It follows that $(\OU,\OV)\in \mathcal{O}(\OX)\cap \mathcal{O}(\OY)$, which means $\OX$ and $\OY$ have {\em simultaneous ordered SVD} \cite{LS05a,LS05b} with orthogonal pair $(\OU,\OV)$ in the sense that 
\begin{equation}\label{SSVD}
    \OX=\OU({\rm Diag}\,\sigma(\OX))\OV^T\qquad \mbox{and}\qquad \OY=\OU({\rm Diag}\,\sigma(\OY))\OV^T,
\end{equation}
where $\sigma(\OX)\eqdef\left (\sigma_1(\OX), \ldots,\sigma_{n_1}(\OX)\right)^T$ containing all the singular values of $\OX$ in decreasing order and 
$${\rm Diag}\,\sigma(\OX)\eqdef\begin{pmatrix}\sigma_1(\OX)&\ldots&0&0&\ldots& 0\\
0&\ddots&0&0&\ldots& 0\\ 
0&\ldots &\sigma_{n_1}(\OX)&0&\ldots 
&0\end{pmatrix}_{n_1\times n_2}.$$

In the spirit of Theorem~\ref{UniRa}, to study the solution uniqueness of problem~\eqref{NNM1} we have to compute the radial cone of the nuclear norm. Its explicit formula below is helpful for further studies in this paper and also for the numerical experiments when checking solution uniqueness in Section~\ref{NE}. 

\begin{Lemma}\label{RC} Let $\OY\in \partial \|\OX\|_*$. Suppose that $\OX$ and $\OY$ have the  simultaneous ordered  SVD with an orthogonal matrix pair $(U,V)\in \mathcal{O}(\OX)\cap \mathcal{O}(\OY)$. Then we have
\begin{equation}\label{Rd}
\mathcal{R}_{\partial g^*(\OY)}(\OX)=\left\{U\begin{pmatrix}A&BC&0\\C^TB^T&C&0\\0&0&0\end{pmatrix}V^T\in \R^{n_1\times n_2}|\; A\in \mathbb{S}^r,B\in \R^{r\times (p-r)}, C\in \mathbb{S}_+^{p-r}\right\},
\end{equation}
where $\mathbb{S}^r$ is the set of $r \times r$ symmetric matrices, $r=\rank(\OX)$ and $p=p(\OY)$ is defined in \eqref{p}.
\end{Lemma}

\noindent{\bf Proof.} To verify the inclusion ``$\subset$'' in \eqref{Rd}, pick any $W\in \mathcal{R}_{\partial g^*(\OY)}(\OX)\setminus \{0\}$. Thus, there exists $\nu>0$ such that $\OX+\nu W\in \partial g^*(\OY)$. By formulae \eqref{OX} and \eqref{Inver}, there exist $A\in \mathbb{S}^r$, $D\in \R^{r\times(p-r)}$, and $C\in \mathbb{S}^{p-r}_+$ such that $W=U\begin{pmatrix}A&D&0\\D^T&C&0\\0&0&0\end{pmatrix}V^T$. As $\partial g^*(\OY)$ is a convex set and $\OX\in \partial g^*(\OY)$, we have $\OX+tW\in \partial g^*(\OY)$ for any $t\in (0,\nu)$. If follows from \eqref{Inver} that
\begin{equation}\label{Pos}
\begin{pmatrix}\Bar\Sigma_r+tA&tD\\tD^T&tC\end{pmatrix}\in \mathbb{S}^p_+.
\end{equation} 

Let $\lm_{\min}:\mathbb{S}^r\to \R$ be the smallest eigenvalue function. It is well-known that this function is $1$-Lipschitz continuous with respect to the Frobenius norm. As $\lm_{\min}(\Bar\Sigma_r)=\sigma_r(\OX)>0$, we have $\lm_{\min}(\Bar\Sigma_r+tA)>0$ for all $t\in [0,\lm_{\min}(\Bar\Sigma_r)/\|A\|_F)$, i.e.,  the matrix $\Bar\Sigma_r+tA$ is positive definite or $\Bar\Sigma_r+tA\succ0$ for all such $t$. We then get by the Schur complement that 
\begin{equation}\label{Schur}
  \begin{pmatrix}\Bar\Sigma_r+tA&tD\\tD^T&tC\end{pmatrix}\succeq0\quad\mbox{iff}\quad C-tD^T(\Bar\Sigma_r+tA)^{-1}D\succeq0;
\end{equation}
see also \cite[Theorem~7.7.6 and its proof]{HJ85}. 
For any $x\in \Ker C$, we derive from \eqref{Pos} and \eqref{Schur} that 
\[
0\le x^T\left(C-tD^T(\Bar\Sigma_r+tA)^{-1}D\right)x=- tx^TD^T(\Bar\Sigma_r+tA)^{-1}Dx,
\]
which implies $Dx=0$ as $\Bar\Sigma_r+tA\succ0$; i.e. $\Ker C \subset \Ker D$. Define $\mathcal{X}\eqdef\{GC|\; G\in \R^{r\times(p-r)}\}$, a subspace of $\R^{r\times(p-r)}$. We show next that $D\in \mathcal{X}$. Indeed, by contradiction suppose that $D\notin \mathcal{X}$, it follows from the separation theorem that there exists a matrix $H\in \R^{r\times(p-r)}$ such that 
\begin{equation}\label{eq:septhm}
\Tr(DH^T)=\la D,H\ra<0=\min_{G\in \R^{r\times(p-r)}}\la GC,H\ra= \min_{G\in \R^{r\times(p-r)}}\Tr(C^TG^TH)=\min_{G\in \R^{r\times(p-r)}}\Tr(G^THC^T),
\end{equation}
which implies that $HC^T=0$, i.e., $CH^T=0$. Thus each row of $H$ is a vector in $\Ker C$, which implies that $DH^T=0$ since we have shown just before that $\Ker C \subset \Ker D$. But this contradicts \eqref{eq:septhm}. Thus $D\in \mathcal{X}$, i.e.,  there exists $B\in \R^{r\times (p-r)}$ such that  $D=BC$. This verifies the inclusion ``$\subset$'' in \eqref{Rd}. 

To justify the converse inclusion ``$\supset$'' in \eqref{Rd}, pick any $W=U\begin{pmatrix}A&BC&0\\C^TB^T&C&0\\0&0&0\end{pmatrix}V^T\in \R^{n_1\times n_2}$ with  $A\in \mathbb{S}^r$, $B\in \R^{r\times (p-r)}$, and  $C\in \mathbb{S}_+^{p-r}$. Again, arguing as above, we have $\lm_{\min}(\Bar\Sigma_r+tA)>0$ for all $t > 0$ small enough, meaning that $\Bar\Sigma_r+tA$ is nonsingular. According to the Schur complement \eqref{Schur} and formula~\eqref{Inver} again, $\OX+tW\in\partial g^*(\OY)$ if and only if 
\[
C-tC^TB^T(\Bar\Sigma_r+tA)^{-1}BC\succeq0.
\]
As $C\in \mathbb{S}^p_+$, it is diagonalizable. Hence we may write  $C=Z\begin{pmatrix}\Lm&0\\0&0\end{pmatrix}Z^T$ with $\Lm\in \R^{s\times s}$ being a diagonal matrix of positive eigenvalues of $C$ for some $s\le p-r$ and $Z\in \R^{(p-r)\times (p-r)}$ being an orthogonal matrix. It follows that $C-tC^TB^T(\Bar\Sigma_r+tA)^{-1}BC\succeq0$ if and only if 
\begin{equation}\label{Pos2}
0\preceq\begin{pmatrix}\Lm&0\\0&0\end{pmatrix}-t\begin{pmatrix}\Lm&0\\0&0\end{pmatrix}Z^TB^T(\Bar\Sigma_r+tA)^{-1}BZ\begin{pmatrix}\Lm&0\\0&0\end{pmatrix}=\begin{pmatrix}\Lm-t\Lm M_t\Lm &0\\0&0\end{pmatrix},
\end{equation}
where $M_t\in \R^{s\times s}$ is the submatrix of the positive semidefinite matrix $Z^TB^T(\Bar\Sigma_r+tA)^{-1}BZ$ with the same position as $\Lm$ in $\begin{pmatrix}\Lm&0\\0&0\end{pmatrix}$. As $\Bar\Sigma_r\succ0$, $Z^TB^T(\Sigma_0+tA)^{-1}BZ$ and $M_t$ are  continuous mappings with respect to $t$  around $0$. Since $M_t$ is symmetric, we can choose $t>0$ small enough such that $\Lm-t\Lm M_t\Lm\succeq 0$. Thus condition~\eqref{Pos2} is valid, which implies  $C-tC^TB^T(\Bar\Sigma_r+tA)^{-1}BC\succeq0$. This together with the Schur complement \eqref{Schur} and formula \eqref{Inver} tells us that $X_0+tW\in \partial g^*(\OY)$, which implies $W\in \mathcal{R}_{\partial g^*(\OY)}(\OX)$. The proof is complete. \endproof

The following result is a direct consequence of Theorem~\ref{UniRa} and formula \eqref{Rd} of the radial cone. 

\begin{Corollary}[Solution uniqueness of low-rank optimization problems]\label{UniRaN} Suppose that $\OX$ is an optimal solution of problem~\eqref{NNM1} with $\OX\in \Phi^{-1}({\rm int}\, (\dom f))$ and $\OY\eqdef-\Phi^*\nabla f(\Phi \OX)\in \partial \|\OX\|_*$. Then $\OX$ is the unique solution of problem~\eqref{NNM1} if and only if 
\begin{equation}\label{NonEq}
\Phi\left(\OU\begin{pmatrix}A&BC&0\\C^TB^T&C&0\\0&0&0\end{pmatrix}\OV^T\right)=0 ~ \mbox{with} ~ A\in \mathbb{S}^r,B\in \R^{r\times (p-r)}, C\in \mathbb{S}_+^{p-r} \quad \Longrightarrow \quad A=C=0 ,
\end{equation}
where the pair $(\OU,\OV)\in \mathcal{O}(\OX)\cap\mathcal{O}(\OY)$ is from~\eqref{SSVD}, $r={\rm rank}\, (\OX)$, and $p=p(\OY)$ from \eqref{p}.   
\end{Corollary}

\begin{Remark}[Closedness of the radial cone for the nuclear norm]\label{ClNu}{\rm  As the nuclear norm satisfies the quadratic growth condition \cite{CDZ17,ZS17}, Corollary~\ref{StrUnq} tells us that  $\OX$ is a unique solution of problem~\eqref{NNM1} if and only if it is a strong solution provided that the radial cone $\mathcal{R}_{\partial g^*(\OY)}(\OX)$ is closed. Note that its closure is the tangent cone \eqref{RT} computed in \cite[Corollary~4.2]{FNP23}
\begin{equation}\label{Tang}
    \mathcal{T}_{\partial g^*(\OY)}(\OX)=\left\{\OU\begin{pmatrix}A&B&0\\B^T&C&0\\0&0&0\end{pmatrix}\OV^T\in \R^{n_1\times n_2}|\; A\in \mathbb{S}^r,B\in \R^{r\times (p-r)}, C\in \mathbb{S}_+^{p-r}\right\};
\end{equation}
see also \cite[Lemma~2.3 and Proposition~2.1]{LPB21} for a different computation and approach. 
If $p=r$, it is obvious that 
\[
\mathcal{R}_{\partial g^*(\OY)}(\OX)=\mathcal{T}_{\partial g^*(\OY)}(\OX)=\OU\begin{pmatrix}\mathbb{S}^r&0\\0&0\end{pmatrix}\OV^T,
\]
which implies that  the radial cone $\mathcal{R}_{\partial g^*(\OY)}(\OX)$ is closed in this case. 

If $p>r$, we show next that the radial cone $\mathcal{R}_{\partial g^*(\OY)}(\OX)$ is not closed. Indeed, pick
\[
W_k=\OU\begin{pmatrix}0&\mathbb{E}&0\\\mathbb{E}^T&\frac{1}{k}\Id_{p-r}&0\\0&0&0\end{pmatrix}\OV^T\to W=\OU\begin{pmatrix}0&\mathbb{E}&0\\\mathbb{E}^T&0&0\\0&0&0\end{pmatrix}\OV^T
\]
as $k\to \infty$, where $\mathbb{E}\in \R^{r\times (p-r)}$ is the matrix of all entries $1$ and $\Id_{p-r}$ is the $(p-r)\times (p-r)$ identity matrix. Note that $W_k\in \mathcal{R}_{\partial g^*(\OY)}(\OX)$, but $W\notin \mathcal{R}_{\partial g^*(\OY)}(\OX)$. Hence the radial cone $\mathcal{R}_{\partial g^*(\OY)}(\OX)$ is not closed. This situation is very different from the case when the regularizer $g$ is $\ell_1/\ell_2$ norm when the corresponding radial cone is always closed; see Remark~\ref{Cl12}.

Thus the radial cone $\mathcal{R}_{\partial g^*(\OY)}(\OX)$ is  closed if and only if $p=r$. This is the case of dual nondegeneracy condition \eqref{DuNon} discussed in Remark~\ref{Cl12}. Indeed, note from \eqref{Inver} that 
\begin{equation}\label{RiSpec}  
{\rm ri}\, (\partial g^*(\OY)) =\OU\begin{pmatrix}\mathbb{S}_{++}^{p} &0\\0 &0\end{pmatrix}\OV^T,
\end{equation}
where $\mathbb{S}_{++}^{p}$ is the set of all symmetric and positive definite $p\times p$ matrices. As $\OX=\OU\begin{pmatrix}\Bar\Sigma_r &0\\0 &0\end{pmatrix}\OV^T$, $p=r$ if and only if $\OX\in {\rm ri}\, (\partial g^*(\OY))$. Note further from~\eqref{subdif} that 
\begin{equation*}
{\rm ri}\, (\partial \|\OX\|_*)=\left\{U\begin{pmatrix} \Id_r&0\\ 0 & W\end{pmatrix}V^T|\; \|W\|< 1\right\}.
\end{equation*}
Thus, the case $p=r$ is also equivalent to the  {\em nondegeneracy condition}
\begin{equation}\label{NonD}
    \OY\in {\rm ri}\, (\partial \|\OX\|_*). 
\end{equation}
Consequently, the radial cone $\mathcal{R}_{\partial g^*(\OY)}(\OX)$ is  closed if and only if the nondegeneracy condition \eqref{NonD} is satisfied. \hfill{$\triangle$}
}
\end{Remark}

\begin{Remark}[Checking condition \eqref{NonEq} via optimization]\label{Max}{\rm Condition~\eqref{NonEq} means that the following optimization problem 
\begin{equation*}\label{NonOp}
\max\; \frac{1}{2}\|A\|^2_F+\frac{1}{2}\|C\|_F^2\; \;\mbox{s.t.}\; \;\Phi\left(\OU\begin{pmatrix}A&BC&0\\C^TB^T&C&0\\0&0&0\end{pmatrix}\OV^T\right)=0, A\in \mathbb{S}^r,B\in \R^{r\times (p-r)}, C\in \mathbb{S}_+^{p-r}
\end{equation*}
has global maximum value $0$ and global maximizers $(0,\overline B,0)\in \mathbb{S}^r\times \R^{r\times (p-r)}\times \mathbb{S}_+^{p-r}$ for any $\overline B$.
This is a nonlinear concave optimization problem.  Unfortunately, most of available solvers will provide local maximizers. However, if $(0,\overline B,0)\in \mathbb{S}^r\times \R^{r\times (p-r)}\times \mathbb{S}_+^{p-r}$ are local maximizers for any $\overline B$, then $0$ is the global maximum value. Indeed, suppose that $(A,B,C)\in \mathbb{S}^r\times \R^{r\times (p-r)}\times \mathbb{S}_+^{p-r}$ satisfies the first equality in \eqref{NonEq}, hence is a feasible point of the above optimization problem. We observe that $(tA,B,tC)$ is also feasible for any $t \neq 0$. By choosing sufficiently small $|t| > 0$,  $(tA,B,tC)$ is in a small neighborhood of $(0,B,0)$. If  $(0,B,0)$ is a local maximizer of the above optimization problem, we have  
\[
t^2\|A\|^2_F+t^2\|C\|^2_F\le 0,
\]
which implies $A=0$ and $C=0$. Hence $0$ is the global maximum value.  \hfill{$\triangle$}
}
\end{Remark}

\begin{Corollary}[Strong minima vs solution uniqueness]\label{StrUnq2} Let $\OX\in \Phi^{-1}({\rm int}\, (\dom f))$ be an optimal solution of problem~\eqref{NNM1}. Suppose that Nondegeneracy Condition \eqref{NonD} is satisfied. Then $\OX$ is a unique solution of problem~\eqref{NNM1} if and only if it is a strong solution. 
\end{Corollary}
\begin{proof} Suppose that Nondegeneracy Condition \eqref{NonD} is satisfied. As discussed at the end of  Remark~\ref{ClNu}, the radial cone $\mathcal{R}_{\partial g^*(\OY)}(\OX)$ is closed. Note also that the nuclear norm satisfies the quadratic growth condition \eqref{QGC} by \cite{ZS17,CDZ17}. By Corollary~\ref{StrUnq}, $\OX$ is a unique solution of problem~\eqref{NNM1} if and only if it is a strong solution.
\end{proof}

When Nondegeneracy Condition \eqref{NonD} is satisfied, it follows from \cite[Theorem~6.3]{LZ13} that $\OX$ is a strong solution of problem~\eqref{NNM1} if and only if it is a {\em tilt stable} solution of in the sense that the optimal solution set
\[
S(V)\eqdef{\rm argmin}\quad \{f(\Phi X)+\|X\|_*-\la V,X\ra|\; X\in \R^{n_1\times n_2}\} \quad \mbox{for}\quad V\in \R^{n_1\times n_2} 
\]
is single-valued and Lipschitz continuous around $0$ with $S(0)=\{\OX\}$; see also the recent result \cite[Corollary~3.10]{HS23} for similar equivalence. Tilt stability was introduced by Poliquin and Rockafellar \cite{PR98}  for general optimization frameworks. This important stability is characterized in \cite{PR98} via the {\em generalized Hessian}/{\em second-order limiting subdifferential} introduced by Mordukhovich \cite{M92}; see also \cite{AG08,DMN14} for different characterizations for tilt stability. Although the second-order limiting subdifferential has rich calculus \cite{M1,MR12}, its computation is nontrivial in many situations including the nuclear norm function in \eqref{NNM1}. By combining Corollary~\ref{StrUnq2} with the aforementioned \cite[Theorem~6.3]{LZ13} we obtain below the equivalence between  solution uniqueness and tilt stability for problem~\eqref{NNM1} under Nondegeneracy Condition \eqref{NonD}. This equivalence is interesting, as we see that the assumption of solution uniqueness is  already included in the definition of tilt stability ($S(0)=\{\OX\})$. Somewhat Nondegeracy Condition~\eqref{NonD} makes the problem stable. 

\begin{Corollary}[Solution uniqueness and tilt stability]\label{Tilt} Suppose that $\OX\in \Phi^{-1}({\rm int}\, (\dom f))$ is  an optimal solution of problem~\eqref{NNM1} and that Nondegeneracy Condition \eqref{NonD} holds. Then $\OX$ is a tilt-stable solution if and only if it is the unique solution.   
\end{Corollary}

\section{Solution uniqueness of convex optimization problems with linear constraints}\label{sec:uniquelin}
\setcounter{equation}{0}

\subsection{General regularizer}
In this section, we study solution uniqueness of the  optimization problem \eqref{P0}
\begin{equation}\label{ConP}
\min_{x\in \XX}\quad g(x)\quad \mbox{subject to}\quad \Phi x=y_0,
\end{equation}
where $g:\XX\to \R$ is a {\em continuous} convex function, $\Phi:\XX\to \YY$ is still a linear operator between two Euclidean spaces, and $y_0$ is a vector in $\YY$. Suppose that $x_0$ is a {\em feasible solution} of problem~\eqref{ConP}, i.e., $\Phi x_0=y_0$. Solution uniqueness of this problem can be characterized  via the so-called  {\em descent cone}  \cite[ Proposition 2.1]{CRPW12} recalled below. 

\begin{Definition}[Descent cone \cite{CRPW12}] The descent cone at $x_0$ is defined by 
\begin{equation}\label{Des}
    \mathcal{D}_g(x_0)\eqdef{\rm cone}\,\{x-x_0|\; g(x)\le g(x_0)\}.
\end{equation}
\end{Definition}

When $g$ is a polyhedral convex function, the set $\{x-x_0|\; g(x)\le g(x_0)\}$ is polyhedral and thus $\mathcal{D}_g(x_0)$ is a closed convex cone.  But it is not closed in general; see, e.g., \cite[Example~4.3]{FNT21} for the case of $\ell_1/\ell_2$ norm. Note that the closure of the descent cone lies in the {\em critical cone} defined by 
\begin{equation}\label{Crit}
    \mathcal{C}(x_0)=\{w\in \mathbb{X}|\;   dg(x_0)(w)\le 0\}
\end{equation}
and the interior of the critical cone is a subset of the descent cone;
see, e.g., \cite[Lemma~3.6]{FNT21}. Hence the differences between them are on the boundary of $\mathcal{C}(x_0)$ denoted by ${\rm bd}\, \mathcal{C}(x_0)$. In the recent paper \cite[Theorem~5.1]{FNT21}, the descent cone of the $\ell_1/\ell_2$ norm is computed exactly via the critical cone and first-order information.  Here, we extend that result to the case of general continuous convex function by involving the radical cone.

\begin{Proposition}[Descent cone via radial cone]\label{DesRd} Suppose that $g:\XX\to \R$ is a continuous convex function. The descent cone $D_g(x_0)$ in \eqref{Des} is computed by
\begin{equation}\label{ComDes}
    D_g(x_0)=\bigcup_{y\in \partial g(x_0)} \left[\mathcal{R}_{\partial g^*(y)}(x_0)\cap {\rm bd}\, \mathcal{C}(x_0)\right]\cup({\rm int}\, \mathcal{C}(x_0)).
\end{equation} 
\end{Proposition}
\begin{proof} Let us start to prove the inclusion ``$\subset $'' in \eqref{ComDes}. Pick any $w\in D_g(x_0)$, there exists some $\tau>0$ such that $g(x_0+\tau w)\le g(x_0)$. As the function $g$ is convex and continuous, we derive from \eqref{Dd} that 
\begin{equation}\label{max}
0\ge g(x_0+\tau w)-g(x_0)\ge \tau\max_{y\in \partial g(x_0)}\la y,w\ra=\tau dg(x_0)(w),
\end{equation}
which implies that $dg(x_0)(w)\le0$.
Note further that $dg(x_0)(\cdot)$ is also a proper convex and continuous function. It follows that 
\[
{\rm int}\, \mathcal{C}(x_0)=\{u\in \XX|\; dg(x_0)(u)<0\}.
\]
If $dg(x_0)(w)<0$, we have $w\in {\rm int}\, \mathcal{C}(x_0)$. If $dg(x_0)(w)=0$, i.e., $w\in {\rm bd}\, \mathcal{C}(x_0)$ (the boundary of $\mathcal{C}(x_0)$), note from \eqref{max}  that $g(x_0+\tau w)=g(x_0)$ and there exists some $y\in \partial g(x_0)$ such that $\la y,w\ra=0$. Hence we have 
\[
g(x_0+\tau w)=g(x_0)=\la y,x_0\ra-g^*(y)=\la y,x_0+\tau w\ra-g^*(y),
\]
which implies that $y\in \partial g(x_0+\tau w)$, i.e., $x_0+\tau w\in \partial g^*(y)$. It follows that $w\in \mathcal{R}_{\partial g^*(y)}(x_0)$. The inclusion ``$\subset$'' in \eqref{ComDes} is verified. 

To justify the converse inclusion ``$\supset $'' in \eqref{ComDes}, we pick any $w$ in the right-hand side. If $w\in {\rm int}\, \mathcal{C}(x_0)$, i.e., $dg(x_0)(w)<0$,  there exists some $\tau>0$ sufficiently small so that $$g(x_0+\tau w)-g(x_0)\le 0,$$ which means $w\in \mathcal{D}_g(x_0)$. If $w\in {\rm bd}\, \mathcal{C}(x_0)$, i.e., $dg(x_0)(w)=0$, there exists some $y\in \partial g(x_0)$ such that $w\in \mathcal{R}_{\partial g^*(y)}(x_0)$. We find some $t>0$ such that $x_0+tw\in \partial g^*(y)$. It follows that   
\[
g(x_0+tw)=\la y, x_0+tw\ra-g^*(y)=g(x_0)+t\la y,w\ra\le g(x_0)+t dg(x_0)(w)=g(x_0),
\]
which also implies that $w\in  \mathcal{D}_g(\ox)$. This verifies the inclusion ``$\supset$'' in \eqref{ComDes}.
\end{proof}

The following result taken from \cite[ Proposition 2.1]{CRPW12} is the well-known characterization of solution uniqueness via the descent cone. 

\begin{Proposition}[Descent cone for solution uniqueness \cite{CRPW12}]\label{DC} A feasible point $x_0$ of problem~\eqref{ConP} is the unique solution if and only if $$\Ker \Phi \cap \mathcal{D}_g(x_0)=\{0\}.$$
\end{Proposition}

The exact computation of the descent cone \eqref{ComDes} gives the impression that it is a big and complicated set. But its intersection with $\Ker \Phi$ seems to be small and simple as in the above result. We enhance this observation by providing next the calculation of this intersection  when $x_0$ is an optimal solution, which means  $0\in \partial g(x_0)+\Im \Phi^*$, i.e.,  
\begin{equation}\label{SC}
\Delta(x_0)\eqdef\partial g(x_0)\cap \Im \Phi^*\neq \emptyset.
\end{equation}
This is known as the Source Condition \cite{GHS11}. An element of $\Delta(x_0)$ is called a {\em dual certificate}.

\begin{Proposition}\label{Int} Let $x_0$ be an optimal solution of problem~\eqref{ConP}. Then we have
\begin{equation}\label{ComDes2}
\Ker \Phi\cap \mathcal{D}_g(x_0)=\Ker \Phi\cap \mathcal{R}_{\partial g^*(y)}(x_0)\quad \mbox{for any}\quad y\in \Delta(x_0). 
\end{equation}
\end{Proposition}
\begin{proof}
As $x_0$ is an optimal solution of problem~\eqref{ConP}, $\Delta(x_0)\neq \emptyset$. Pick any dual certificate $y\in \Delta(x_0)$. To justify the ``$\subset$" inclusion in \eqref{ComDes2}, take any $w\in \Ker \Phi\cap \mathcal{D}_g(x_0)$, we find some $\tau>0$ such that $g(x_0+\tau w)\le g(x_0)$. Since $\Phi(x_0+\tau w)=\Phi x_0=y_0$, we have $g(x_0+\tau w)=g(x_0)$. Note also that $\la y,w\ra=0$ as $y\in \Im \Phi^*$ and $w\in \Ker \Phi$. It follows that 
\[
g(x_0+\tau w)=g(x_0)=\la y,x_0\ra-g^*(y)=\la y,x_0+\tau w\ra-g^*(y),
\]
which implies that $x_0+\tau w\in \partial g^*(y)$, i.e., $w\in \mathcal{R}_{\partial g^*(y)}(x_0)$. This verifies inclusion ``$\subset$'' in \eqref{ComDes2}.

Conversely, pick any $w\in \Ker \Phi\cap \mathcal{R}_{\partial g^*(y)}(x_0)$. We find $t>0$ such that $x_0+tw\in \partial g^*(y)$. It follows that 
\[
g(x_0+tw)=\la y,x_0+tw\ra-g^*(y)=\la y,x_0\ra-g^*(y)=g(x_0), 
\]
which implies  $w\in \mathcal{D}_g(x_0)$ and verifies the converse inclusion ``$\supset$'' in \eqref{ComDes2}. 
\end{proof}

Combining this result with Proposition~\ref{DC}, we obtain a characterization of solution uniqueness for problem~\eqref{ConP} via the radical cone, which is similar to Theorem~\ref{UniRa} in Section~3.

\begin{Theorem}[Characterizations for solution uniqueness via radial cones]\label{UniRd2} The following are equivalent: 
\begin{itemize}
    \item[{\rm(i)}] $x_0$ is a unique optimal solution of  problem \eqref{ConP}.
    \item[{\rm(ii)}] Source Condition~\eqref{SC} holds and  for any dual certificate  $y\in \Delta(x_0)$, one has  
    \begin{equation}\label{P2}
        \Ker \Phi\cap \mathcal{R}_{\partial g^*(y)}(x_0)=\{0\}.
    \end{equation}
    \item[{\rm(iii)}] Source Condition~\eqref{SC} holds and there exists a dual certificate  $\oy\in \Delta(x_0)$ such that 
    \begin{equation}\label{P1}
        \Ker \Phi\cap \mathcal{R}_{\partial g^*(\oy)}(x_0)=\{0\}.
    \end{equation}
\end{itemize}
\end{Theorem}
\noindent{\bf Proof.} Let us start by proving the implication [(i)$\Longrightarrow$(ii)] by supposing that $x_0$ is a unique optimal solution of problem~\eqref{ConP}. Hence Source Condition~\eqref{SC} holds. Moreover, \eqref{P2} comes from Proposition~\ref{DC} and Proposition~\ref{Int}. This verifies (ii).

The implication [(ii)$\Longrightarrow$(iii)] is trivial. Let us justify [(iii)$\Longrightarrow$(i)] by supposing that there exists some $\oy\in \Delta(x_0)\neq \emptyset$ such that  condition \eqref{P1} holds. As Source Condition holds, $x_0$ is an optimal solution of problem~\eqref{ConP}. By Proposition~\ref{DC} and Proposition~\ref{Int} again, $x_0$ is the unique solution.  \endproof

\begin{Remark}\label{LPB}{\rm The equivalence between (i) and (iii) in Theorem~\ref{UniRd2} was established recently in \cite[Lemma~3.1]{LPB21} when $g$ is the nuclear norm; see our Corollary~\ref{UniRaNu} for further details about this case. This equivalence is also obtained in \cite[Theorem~3.1]{HKS23} when $g$ is a gauge function.  The appearance of (ii) in the above result is crucial. It tells us that when condition \eqref{P1} holds at some $\oy\in \Delta(x_0)$, it is also satisfied at any other dual certificate. Hence, the choice of $\oy$ in condition \eqref{P1} is not important. If we know one dual certificate of problem~\eqref{ConP}, we only need to check condition~\eqref{P1} at this dual certificate for solution uniqueness. Such a dual certificate can be found via solving a ``gauge'' convex optimization problem; see 
\cite[Section~4]{FNT21} for more details. It is possible to prove Theorem~\ref{UniRd2} directly without using Proposition~\ref{DC}, but our approach here highlights the connection between the descent cone and the radial cone for solution uniquessness. \hfill$\triangle$

}
\end{Remark}

\begin{Remark}[Strong Minima vs Solution Uniqueness]\label{StroUniq} {\rm When the function $g$ satisfies the quadratic growth condition at $x_0$ for any $y\in \partial g(x_0)$ and it is {\em second-order regular} at $x_0$ in the sense of \cite[Definition~3.85]{BS00}, it is shown recently in \cite[Theorem~3.5]{FNP23} that $x_0$ is a strong solution if and only if 
\begin{equation}\label{Stro}
    \left[\bigcap_{y\in \Delta(x_0)} \mathcal{T}_{\partial g^*(y)}(x_0)\right]\cap \Ker \Phi \cap \mathcal{C}(x_0)=\{0\}.
\end{equation}
If $x_0$ is a strong solution of problem~\eqref{ConP}, it is certainly a unique solution. Conversely, if $x_0$ is a unique solution of problem~\eqref{ConP} and the radial cone $\mathcal{R}_{\partial g^*(y)}(x_0)$ is closed at some dual certificate $y\in \Delta(x_0)$,  Theorem~\ref{UniRd2} tells us that 
\[
0= \Ker \Phi \cap \mathcal{R}_{\partial g^*(y)}(x_0)=\Ker \Phi \cap \mathcal{T}_{\partial g^*(y)}(x_0),
\]
which clearly verifies \eqref{Stro}. Thus $x_0$ is a strong solution of problem~\eqref{ConP} if and only if it is the unique solution provided that  the function $g$ satisfies the quadratic growth condition at $x_0$ for any $y\in \partial g(x_0)$ and the second-order regularity at $x_0$, and that the radial cone $\mathcal{R}_{\partial g^*(y)}(x_0)$ is closed at some dual certificate $y\in \Delta(x_0)$. 

The $\ell_1/\ell_2$ norm satisfies all the three latter conditions; see also Remark~\ref{Cl12}. Hence, $x_0$ is a strong solution of problem ~\eqref{ConP} if and only if it is the unique solution in this case. This fact was  established recently in \cite[Thereom~5.1]{FNT21} by a different approach. 

It is also known that the nuclear norm $g(\cdot)=\|\cdot\|_*$ satisfies the quadratic growth condition at $X_0\in \R^{n_1\times n_2}$ for any $Y\in \partial \|X_0\|_*$ and  second-order regularity at $X_0$; see \cite{CDZ17,ZS17}, but the radical cone  $\mathcal{R}_{\partial g^*(Y)}(X_0)$ is not always closed. As discussed in Remark~\ref{ClNu}, this radial cone 
is closed if and only if 
$Y\in {\rm ri}\, \partial \|X_0\|_*$. 
The existence of such dual certificate $Y$  means 
\begin{equation}\label{NSC}
    \Im \Phi^*\cap {\rm ri}\, \partial \|X_0\|_*\neq \emptyset,
\end{equation}
which is known as the {\em Nondegeneracy Source Condition} introduced in \cite{CR09,CR13} for nuclear norm minimization problems. We will add some details about this case in Corollary~\ref{StroUniNu}. \hfill$\triangle$
}
\end{Remark}



\subsection{Nuclear norm case}
Next let us consider the nuclear norm optimization problem:
\begin{equation}\label{NNM}
\min_{X\in \R^{n_1\times n_2}}\quad \|X\|_*\quad \mbox{subject to}\quad \Phi X=M_0,    
\end{equation}
where $\Phi:\R^{n_1\times n_2}\to \R^m$ is a linear operator ($n_2\ge n_1$) and $M_0\in \R^m$ is some known vector. The following characterizations of solution uniqueness for nuclear norm minimization problem~\eqref{NNM} come directly from Theorem~\ref{UniRd2} and the calculation of the radial cone in \eqref{Rd}.

\begin{Corollary}[Solution uniqueness of nuclear norm minimization problem] \label{UniRaNu} Let $X_0$ be an optimal solution of problem~\eqref{NNM}. The following are equivalent 

\begin{itemize}
    \item[{\rm(i)}] $X_0$ is a unique solution of problem~\eqref{NNM}.
    \item[{\rm(ii)}] For any dual certificate $Y\in \Delta(X_0)$, the system 
\begin{equation}\label{Ker}
\Phi\left(U\begin{pmatrix}A&BC&0\\C^TB^T&C&0\\0&0&0\end{pmatrix}V^T\right)=0 \quad \mbox{with}\quad A\in \mathbb{S}^r,B\in \R^{r\times (p-r)}, C\in \mathbb{S}_+^{p-r}
\end{equation}
 gives $A=0$ and $C=0$, where the pair $(U,V)\in \mathcal{O}(X_0)\cap\mathcal{O}(Y)$ comes from the simultaneous ordered singular value decomposition of $X_0$ and $Y$, $r={\rm rank}\, (X_0)$, and $p=p(Y)$ from \eqref{p}.  
\item[{\rm(iii)}] Condition~{\rm (ii)} holds at some dual certificate $\OY\in \Delta(X_0)$. 
 \end{itemize}
\end{Corollary}

\begin{Remark}\label{ComY}{\rm As discussed in Remark~\ref{LPB}, solution uniqueness for nuclear norm minimization problem \eqref{NNM} is characterized via the radical cone in \cite[Lemma~3.1]{LPB21}. The above result advances  \cite[Lemma~3.1]{LPB21} in some ways. The most important thing is that the radical cone is fully calculated, i.e., the general condition~\eqref{P1} is explicitly expressed by condition~\eqref{Ker} at some dual certificate. This condition can be written  and solved as a nonlinear optimization problem; see our previous Remark~\ref{Max}. Moreover, we can pick any dual certificate $Y\in \Delta(X_0)$ and check  \eqref{Ker} at $Y$ for solution uniqueness at $X_0$. A particular dual certificate of problem~\eqref{NNM} can be found by solving a {\em spectral norm minimization problem} \cite[Remark~5.5]{FNT21}. Specifically, suppose that the compact SVD of $X_0$ is
\begin{equation}\label{X0}
 X_0= U_0\Sigma_0V_0^T,
\end{equation}
where $\Sigma_0$ is the diagonal matrix including all the positive singular values of $X_0$ with decreasing order and ${\rm rank}\, (\Sigma_0)=r$, and $U_0\in \R^{n_1\times r}$ and $V_0\in \R^{n_2\times r}$ are orthogonal matrices.
Recall the {\em model tangent space}  \cite{CR13} at $X_0$: 
\begin{equation}\label{T}
    \mathbb{T}\eqdef\{U_0Y^T+XV_0^T|\; X\in \R^{n_1\times r}, Y\in  \R^{n_2\times r}\}.
\end{equation}
Consider the following spectral norm minimization problem
\begin{equation}\label{Spec}
    \min_{Z\in \mathbb{T}^\perp}\quad \|Z\|\quad\mbox{subject to}\quad \mathcal{N}Z=-\mathcal{N} E_0\quad\mbox{with}\quad E_0=U_0V_0^T,
\end{equation}
where $\mathcal{N}$ is a linear operator satisfying $\Ker \mathcal{N}=\Im \Phi^*$. When $X_0$ is an optimal solution of problem~\eqref{NNM},  \cite[Remark~5.5]{FNT21} shows that  the optimal value of problem~\eqref{Spec} denoted by $\rho(X_0)$ is smaller than or equal to $1$. Moreover, if $Z_0$ is an optimal solution of problem~\eqref{Spec}, $Z_0+E_0$ is a dual certificate of problem~\eqref{NNM}; see our Section~\ref{NE} for further details on Numerical Experiment. \hfill$\triangle$

}
\end{Remark}

The following result is similar to Corollary~\ref{StrUnq2}, at which we provide the equivalence between solution uniqueness and strong minima under Nondegenerate Source Condition \eqref{NSC}. It is already discussed in Remark~\ref{StroUniq}. 

\begin{Corollary}\label{StroUniNu} Suppose that $X_0$ is an optimal solution of problem~\eqref{NNM} and Nondegenerate Source Condition \eqref{NSC} is satisfied. Then $X_0$ is a strong solution of problem~\eqref{NNM} if and only if it is the unique solution. 
\end{Corollary}
\begin{proof} As discussed in Remark~\ref{StroUniq}, Nondegenerate Source Condition \eqref{NSC} guarantees  closedness of the radial cone $\mathcal{R}_{\partial g^*(Y)}(X_0)$ for some dual certificate $Y\in \Delta(X_0)$. It follows that 
\[
\mathcal{R}_{\partial g^*(Y)}(X_0)=\mathcal{T}_{\partial g^*(Y)}(X_0). 
\]
Hence we have 
\[
\Ker \Phi \cap \mathcal{R}_{\partial g^*(Y)}(X_0)=\Ker \Phi \cap \mathcal{T}_{\partial g^*(Y)}(X_0). 
\]
By Theorem~\ref{UniRd2} and \cite[Theorem~5.2]{FNP23}, $X_0$ is a strong solution of problem~\eqref{NNM} if and only if it is the unique solution. 
\end{proof}

 
 The following result gives a necessary condition for solution uniqueness. 

\begin{Corollary}\label{SRICo} Let $X_0$ be an optimal solution of problem~\eqref{NNM}. If $X_0$ is a unique solution, then the following {\em Strict Restricted Injectivity} holds
\begin{equation}\label{SRI}
\Ker \Phi\cap U_0\mathbb{S}^rV_0^T=\{0\}. 
\end{equation}
If additionally Nondegeneracy Source Condition~\eqref{NSC} holds, then Strict Restricted Injectivity is also sufficient  for solution uniqueness at $X_0$. 
\end{Corollary}
\begin{proof}
    Note from the formula of radial cone \eqref{Rd} that 
    \[
    U_0\mathbb{S}^rV_0^T\subset \mathcal{R}_{\partial g^*(Y)}(X_0)\quad \mbox{for any}\quad Y\in \Delta(X_0). 
    \]
If $X_0$ a unique solution of problem~\eqref{NNM}, we get from Theorem~\ref{UniRd2} that  Strict Restricted Injectivity~\eqref{SRI} is satisfied. 

Conversely, if Nondegeneracy Source Condition~\eqref{NSC} holds, there exists $Y\in \Delta(X_0)$ such that $Y\in {\rm ri}\,\partial \|X_0\|$, which means that $p(Y)=r$. By  the formula of radial cone \eqref{Rd} again, we have
\[
\mathcal{R}_{\partial g^*(Y)}(X_0)= U_0\mathbb{S}^rV_0^T.
\]
It follows from  Theorem~\ref{UniRd2} that Strict Restricted Injectivity \eqref{SRI} is also sufficient for solution uniqueness at $X_0$. 
\end{proof}

Strict Restricted Injectivity is introduced recently in \cite[Corollary~5.7]{FNP23} for nuclear norm minimization problem~\eqref{NNM}, as a necessary condition for strong minima at $X_0$.  The above result is quite similar, but slightly stronger as a unique solution of problem ~\eqref{NNM} may be not a strong one. The following example taken from \cite[Example~5.11]{FNP23} shows such difference.

\begin{Example}[Difference between unique solution and strong solution in problem \eqref{NNM}]\label{Last}  { \rm Consider the following optimization problem 
\begin{equation}\label{Impo}
\min_{X\in\R^{2\times 2}} \|X\|_* \quad\text{ subject to }\quad \Phi X\eqdef\begin{pmatrix}
     X_{11}+X_{22}  \\
     X_{12}-X_{21}+X_{22} 
\end{pmatrix}=\begin{pmatrix}
     1  \\
     0
\end{pmatrix}.
\end{equation}
Note that $X_0=\begin{pmatrix}
     1 &0  \\
     0 &0
\end{pmatrix}$ is an optimal solution of problem~\eqref{Impo}. Indeed, it is easy to see that 
\[
\Im \Phi^*={\rm span}\left\{\begin{pmatrix}1&0\\0&1\end{pmatrix},\begin{pmatrix}0&1\\-1&1\end{pmatrix}\right\}\qquad \mbox{and}\qquad \partial \|X_0\|_*=\begin{pmatrix}
     1 &0  \\
     0 &[-1,1]
\end{pmatrix}.
\]
It follows that 
\[
\Delta(X_0)=\Im \Phi^*\cap \partial \|X_0\|_*=\begin{pmatrix}1&0\\0&1\end{pmatrix}\quad \mbox{and}\quad \Im \Phi^*\cap{\rm ri}\, \partial \|X_0\|_*=\emptyset. 
\]
Thus $X_0$ is a solution, but Nondegeneracy Source Condition~\eqref{NSC} is not satisfied. In the view of Remark~\ref{StroUniq}, the latter is well-expected, as we  want $X_0$ to be a unique solution, but not a strong solution in this example.

For $Y_0=\Id_2\in \Delta(X_0)$, we obtain the radial cone from \eqref{Rd} and the tangent cone from \eqref{Tang} respectively 
\begin{eqnarray}\label{RT2}
\mathcal{R}_{\partial g^*(Y_0)}(X_0)=\left\{\begin{pmatrix}a&bc\\bc&c\end{pmatrix}|\;a,b\in \R, c\ge 0\right\}\mbox{and}\;  \mathcal{T}_{\partial g^*(Y_0)}(X_0)=\left\{\begin{pmatrix}a&b\\b&c\end{pmatrix}|\;a,b\in \R, c\ge 0\right\}.
\end{eqnarray}

To check condition~\eqref{Ker} or \eqref{P2}, due to~\eqref{RT2} we solve the system $\Phi\begin{pmatrix}a&bc\\bc&c\end{pmatrix}=0$ with $c\ge 0$, which means
\[
a+c=0\quad\mbox{and}\quad bc-bc+c=0. 
\]
This clearly gives us $a=c=0$, i.e., condition~\eqref{Ker} is satisfied. Thus, $X_0$ is a unique solution of problem~\eqref{Impo}.

To check whether $X_0$ is a strong solution, we have to verify the condition
\[
\Ker \Phi\cap \mathcal{T}_{\partial g^*(Y_0)}(X_0)=\{0\}
\]
according to \cite[Theorem~5.2]{FNP23}. By \eqref{RT2}, we solve the system $\Phi\begin{pmatrix}a&b\\b&c\end{pmatrix}=0$ with $c\ge 0$:
\[
a+c=0\quad\mbox{and}\quad b-b+c=0\quad \mbox{for}\quad a,b\in \R, c\ge 0, 
\]
which gives us $a=c=0$ and $b\in \R$. Hence the above condition for strong minima at $X_0$ fails, i.e., $X_0$ is not a strong solution of problem~\eqref{Impo}.\hfill$\triangle$
}
\end{Example}

\begin{Remark}[Comparisons with \cite{HS23}] {\rm  The main advantages of using the descent cone in Proposition~\ref{DC} over the radial cone in Theorem~\ref{UniRd2} for solution uniqueness  are that we do not need to know if $X_0$ is an optimal solution of problem~\ref{NNM} initially and do not need to find a dual certificate. But the descent cone seems to be bigger and more complicate than the radial cone due to\eqref{ComDes}. In the recent paper \cite{HP23}, Hoheisel and Paquette introduced the following set 
\begin{equation}\label{W}
W(X_0)=\left\{UM\begin{pmatrix}\Id_r &0\\0&R\end{pmatrix} V^T\left|\;\begin{array}{ll} &M\in \mathbb{S}^{n_1}_+-\mathcal{F}^{n_1,r},\, {\rm Tr}\, (M)=0, \\
&R\in \mathcal{V}_{n_1-r,n_2-r},\,  M\begin{pmatrix}\Id_r &0\\0&RR^T\end{pmatrix}=M\end{array}\right.\right\},
\end{equation}
where $(U,V)\in \mathcal{O}(X_0)$, the set $\mathcal{F}^{n_1,r}$ is defined by
\[
\mathcal{F}^{n_1,r}\eqdef\left\{\begin{pmatrix}A&0\\0&0\end{pmatrix}\in \mathbb{S}^{n_1}_+|\; A\in \mathbb{S}^r_{++}\right\},
\]
and $\mathcal{V}_{n_1-r,n_2-r}$ is the {\em Stiefel manifold}:
\[
\mathcal{V}_{n_1-r,n_2-r}\eqdef\{R\in \R^{(n_1-r)\times (n_2-r)}|\; R^TR=\Id\}.
\]
In particular \cite[Corollary~4.1]{HP23} shows that $X_0$ is a unique solution of problem~\eqref{NNM} if and only if $X_0$ is an optimal solution and 
\begin{equation}\label{HP1}
    \Ker \Phi \cap W(X_0)=\{0\}. 
\end{equation}
The (possibly  nonconvex) set $W(X_0)$ seems to be more complicated than our radial set \eqref{Rd}.  Computing $W(X_0)$ seems to be complicated in high dimensions, but its span \cite[Proposition~3.2]{HP23} is rather simple 
\begin{eqnarray}\label{SpanW}
    {\rm span}\, W(X_0)=\left\{U\begin{pmatrix}
A&B\\C&D\end{pmatrix}V^T|\; A\in \mathbb{S}^r, B\in \R^{r\times (n_2-r)}, C\in \R^{(n_1-r)\times r}, D\in \R^{(n_1-r)\times (n_2-r)}\right\}.
\end{eqnarray}
Thus a sufficient condition for solution uniqueness at $X_0$ is 
\begin{equation}\label{HP2}
    \Ker \Phi \cap {\rm span}\,W(X_0)=\{0\}; 
\end{equation}
see \cite[Corollary~4.2]{HP23}. This fact can be obtained directly from Theorem~\ref{UniRd2} as 
\[
\mathcal{R}_{\partial g^*(Y)}(X_0)\subset {\rm span}\,W(X_0)
\]
for any $Y\in\Delta(X_0)$.  The condition~\eqref{HP2} also shows that $X_0$ is a strong solution of problem~\eqref{NNM} according to \cite[Theorem~5.2]{FNP23}. In our  Example~\ref{Last} later, condition~\eqref{HP2} fails, but $X_0$ is still a unique solution. 

When $X_0$ is an optimal solution of problem~\eqref{NNM} as assumed in \cite[Corollary~4.1]{HP23}, observe from the proof of \cite[Theorem~3.2]{HP23} that 
\begin{equation}\label{LR}
\Ker \Phi \cap W(X_0)=\Ker \Phi\cap  {\rm cone}\,\{X-X_0|\; \|X\|_*=\|X_0\|_*\}=  \Ker \Phi\cap\mathcal{D}(X_0),    
\end{equation}
where $\mathcal{D}(X_0)$ is the descent cone of nuclear norm at $X_0$. 
This together Proposition~\ref{DC} explains why  condition~\eqref{HP1} ensures the solution uniqueness at $X_0$.

Although our condition~\eqref{Ker} depends on the existence of a dual certificate when $X_0$ is an optimal solution, it is more computable and verifiable than \eqref{HP1}; see also our Section~\ref{NE}. In the framework of Example~\ref{Last}, the set $W(X_0)$ in \eqref{W} is computed by
\[\begin{array}{ll}
W(X_0)=\left\{M\begin{pmatrix}1&0\\0&r\end{pmatrix}|\;M\in \mathbb{S}^2_+-\mathcal{F}^{2,1}, {\rm Tr}\,(M)=0, r\in \{-1,1\}  \right\}.
\end{array}
\]
For any $M\in \mathbb{S}^2_+-\mathcal{F}^{2,1}$ with ${\rm Tr}\,(M)=0$, we write 
\[
M=\begin{pmatrix}a&b\\b&c\end{pmatrix}-\begin{pmatrix}d&0\\0&0\end{pmatrix}=\begin{pmatrix}a-d&b\\b&c\end{pmatrix}\quad \mbox{with}\quad \begin{pmatrix}a&b\\b&c\end{pmatrix}\in \mathbb{S}_+^2\quad \mbox{and}\quad d>0.
\]
As ${\rm Tr}\,(M)=0$, we have $d=a+c>0$. Thus 
\[
M=\begin{pmatrix}-c&b\\b&c\end{pmatrix}
\]
Since $\begin{pmatrix}a&b\\b&c\end{pmatrix}\in \mathbb{S}^2$ and $a+c>0$, if $c=0$ then $a>0$ and  $b=0$. If $c>0$, $b$ is arbitrary, as $a$ can be chosen sufficiently large. Hence the set of all matrices $M$ is  
\[
\begin{pmatrix}0&0\\0&0\end{pmatrix}\cup\left\{\begin{pmatrix}-c&b\\b&c\end{pmatrix}|\; c>0, b\in \R
\right\}.
\]
It follows that 
\[
W(X_0)=\begin{pmatrix}0&0\\0&0\end{pmatrix}\cup\left\{\begin{pmatrix}-c&b\\b&c\end{pmatrix}|\; c>0, b\in \R
\right\}\cup \left\{\begin{pmatrix}-c&-b\\b&-c\end{pmatrix}|\; c>0, b\in \R
\right\}
\]
This set is nonconvex and completely different from the radial cone $\mathcal{R}_{\partial g^*(Y_0)}(X_0)$ obtained in \eqref{RT2}. Moreover, $\mathcal{R}_{\partial g^*(Y_0)}(X_0)$ is not a subset of $W(X_0)$ and vice versa. This observation highlights our distinct approach for solution uniqueness from \cite{HP23}. Nonconvexity of the set $W(X_0)$ may be a challenge when checking condition \eqref{HP1} in problems with higher dimension. \hfill$\triangle$

}    
\end{Remark}

\begin{Remark}[Solution Uniqueness and Sharp Minima]{\rm 
In \cite{CR13} Cand\`es and Retch shown that $X_0$ is a unique solution of problem~\eqref{NNM} if the following {\em Restricted Injectivity}
\begin{equation}\label{RI}
\Ker \Phi \cap\mathbb{T}=\{0\}
\end{equation}
together with Nondegeneracy Source Condition~\eqref{NSC}
 holds. In the spirit of Theorem~\ref{UniRd2}, this fact is explained as follows. When Restricted Injectivity~\eqref{RI} and Nondegeneracy Source Condition~\eqref{NSC} are satisfied,  Strict Restricted Injectivity~\eqref{SRI} holds as $U_0\mathbb{S}^rV_0^T\subset \mathbb{T}$. 
By Corollary~\ref{SRICo}, $X_0$ is the unique solution of problem~\eqref{NNM}. In \cite{HP23}, Hoheisel and Paquette also use the dual version of Restricted Injectivity:
\[
\Im\Phi^*+\mathbb{T}^\perp=\R^{n_1\times n_2}
\]
together with Nondegeneracy Source Condition~\eqref{NSC} as sufficient conditions of solution uniqueness.

It is shown in \cite[Theorem~4.6]{FNT21} that the combination of Restricted Injectivity~\eqref{RI} and Nondegeneracy Source Condition~\eqref{NSC} indeed characterizes  the stronger property,  {\em sharp minima} at $X_0$ \cite{C78,P79}, which means that  there exists some $c>0$ such that 
\[
 \|X\|_*-\|X_0\|_*\ge c\|X-X_0\|_F\quad \mbox{for any \; $X\in \R^{n_1\times n_2}$\; satisfying}\quad \Phi X=M_0.  
\]
In Example~\ref{Last}, we show that a unique solution is not a strong solution, thus not a sharp solution too. Even if Nondegeneracy Source Condition ~\eqref{NSC} holds, solution uniqueness and sharp minima of nuclear norm minimization are different; see, e.g.,  \cite[Example~5.10]{FNP23}. \hfill$\triangle$

 }
\end{Remark}

The next corollary gives the minimum bound for exact recovery. It slightly enhances the recent \cite[Corollary~5.8]{FNP23} by replacing the requirement of strong minima there by solution uniqueness. Its proof based on Strict Restricted Injectivity ~\eqref{SRI} is completely the same; so we omit the detail.   

\begin{Corollary}[Minimum bound for  exact recovery] Suppose that $X_0$ is an $n_1\times n_2$ matrix with rank $r$. Then one needs at least $\frac{1}{2}r(r+1)$ measurements $m$ of $M_0$ so that solving the nuclear norm minimization problem \eqref{NNM} recovers exactly the unique solution $X_0$.

Moreover, there exist infinitely many linear operators $\Phi:\R^{n_1\times n_2}\to \R^{\frac{1}{2}r(r+1)}$ such that $X_0$ is the unique solution of problem \eqref{NNM}. 
\end{Corollary}

The rest of the section is devoted to the computation of the descent cone of the nuclear norm via its first-order information, as it is an important notion used in \cite{CRPW12,ALMT14} to find statistical dimensions in the theory of exact recovery. According to formula~\eqref{ComDes}, we also need to know the critical cone \eqref{Crit} of the nuclear norm:
\begin{equation}\label{CC}
    \mathcal{C}(X_0)= \{W\in \R^{n_1\times n_2}|\; dg(X_0)(W)\le 0\}.
\end{equation}
For any $(U,V)\in \mathcal{O}(X_0)$ and $W\in \R^{n_1\times n_2}$, we write 
\[
W=U\begin{pmatrix}A&B\\C&D\end{pmatrix}V^T
\]
with some $A\in \R^{r\times r}$, $B\in \R^{r\times (n_2-r)}$, $C\in \R^{n_1\times r}$, and $D\in \R^{(n_1-r)\times (n_2-r)}$. By formula \eqref{subdif} and \eqref{Dd}, we have
\[
dg(X_0)(W)=\max_{Y\in \partial \|X_0\|_*}\la Y,W\ra={\rm Tr}\,A+\|D\|_*.
\]
Combining this with \eqref{CC} gives us that 
\begin{eqnarray}\label{CC2}
   \begin{array}{ll} \bd\mathcal{C}(X_0)= \Big\{U\begin{pmatrix}A&B\\C&D\end{pmatrix}V^T|\;& A\in \R^{r\times r}, B\in \R^{r\times (n_2-r)},C\in \R^{n_1\times r}, D\in \R^{(n_1-r)\times (n_2-r)},\\
    &{\rm Tr}\,A+\|D\|_*=0\Big\}.
    \end{array}
\end{eqnarray}
Moreover, we have
\begin{eqnarray*}
   \begin{array}{ll} {\rm int}\, \mathcal{C}(X_0)= \Big\{U\begin{pmatrix}A&B\\C&D\end{pmatrix}V^T|\;& A\in \R^{r\times r}, B\in \R^{r\times (n_2-r)},C\in \R^{n_1\times r}, D\in \R^{(n_1-r)\times (n_2-r)},\\
    &{\rm Tr}\,A+\|D\|_*<0\Big\}.
    \end{array}
\end{eqnarray*}


The following result gives an exact computation of the descent cone for the nuclear norm, which is similar with its formula for $\ell_1/\ell_2$ norm established in \cite[Theorem~5.1]{FNT21}.  

\begin{Proposition}[Descent cone of nuclear norm]\label{US} Let $U_I\Sigma V_I^T$ be the SVD of $X_0$. The descent cone at $X_0$ is calculated by
\begin{equation}\label{DC3}
\mathcal{D}(X_0)=\mathcal{E} \cup ({\rm int}\, \mathcal{C}(X_0)),
\end{equation}
where 
\begin{eqnarray}\label{E}
\begin{array}{ll}\mathcal{E}=\Big\{W\in \R^{n_1\times n_2}|& W=U\begin{pmatrix}A&BC&0\\C^TB^T&C&0\end{pmatrix}V^T, A\in \mathbb{S}^r,B\in \R^{r\times (n_1-r)}, C\in \mathbb{S}_+^{n_1-r},\\
& {\rm Tr}\, (A+C)=0, 
(U,V)\in \mathcal{O}(X_0)\Big\}.
\end{array}
\end{eqnarray}
Moreover, if $X_0$ is an optimal solution of problem \eqref{NNM}, we have
\begin{equation*}
    \Ker \Phi \cap \mathcal{D}(X_0)=\Ker \Phi\cap \left\{U\begin{pmatrix}A&BC&0\\C^TB^T&C&0\\0&0&0\end{pmatrix}V^T|\; A\in \mathbb{S}^r,B\in \R^{r\times (p-r)}, C\in \mathbb{S}_+^{p-r}, {\rm Tr}\, (A+C)=0 \right\} 
\end{equation*}
for any dual certificate $Y\in \Delta(X_0)$ with $p=p(Y)$ from \eqref{p} and any pair $(U,V)\in \mathcal{O}(X_0)\cap\mathcal{O}(Y_0)$. 
\end{Proposition}
\begin{proof}
Let us verify the inclusion ``$\subset$'' in \eqref{DC3}. Pick any $W\in \mathcal{D}(X_0)\setminus {\rm int}\, \mathcal{C}(X_0)$, by formulae \eqref{DC} and \eqref{Rd} there exist $Y\in \partial \|X_0\|_*$ and a pair $(U,V)\in \mathcal{O}(X_0)\cap\mathcal{O}(Y_0)$ such that 
\[
W=U\begin{pmatrix}A&BC&0\\C^TB^T&C&0\\0&0&0\end{pmatrix}V^T\in \bd \mathcal{C}(X_0)
\]
with some $A\in \mathbb{S}^r$, $B\in \R^{r\times (p-r)}$, $C\in \mathbb{S}_+^{p-r}$, $p=p(Y)$ defined in \eqref{p}.    Define 
\[
\overline C=\begin{pmatrix}C&0\\0&0\end{pmatrix}\in \mathbb{S}^{n_1-r}_+\quad \mbox{and}\quad \overline B=\begin{pmatrix}B&0\end{pmatrix} \in \R^{r\times (n_1-r)}.
\]
We obtain that $W=U\begin{pmatrix}A&\overline B\, \overline C&0\\\overline C^T\overline B^T&\overline{C}&0\end{pmatrix}V^T\in \bd \mathcal{C}(X_0)$. Note from \eqref{CC2} that 
\[
0={\rm Tr}\, A+\|\overline C\|_* ={\rm Tr}\, A+{\rm Tr}\, \overline C. 
\]
This tells us that $W\in \mathcal{E}$ defined in \eqref{E} and verifies the inclusion ``$\subset$'' in \eqref{DC3}. 

To justifies the converse inclusion in \eqref{DC3}. Pick any $W\in \mathcal{E}$ with representation
\[
W=U\begin{pmatrix}A&BC&0\\C^TB^T&C&0\end{pmatrix}V^T
\]
with some $A\in \mathbb{S}^r$, $B\in \R^{r\times (n_1-r)}$, and  $C\in \mathbb{S}_+^{n_1-r}$ such that ${\rm Tr}\, (A+C)=0$. By \eqref{CC2}, $W\in \bd \mathcal{C}(X_0)$. Moreover, define $Y=U\begin{pmatrix}\Id_{n_1}&0\end{pmatrix}V^T\in \partial \|X_0\|_*$ by \eqref{subdif}. Note from formula \eqref{Rd}, $W\in \mathcal{R}_{\partial g^*(Y)}(X_0)$. It follows that $\mathcal{E}\subset \mathcal{R}_{\partial g^*(Y)}(X_0)\cap \bd \mathcal{C}(X_0)$. By \eqref{ComDes}, we have the inclusion ``$\supset$'' in \eqref{DC3}. The equality in \eqref{DC3} is verified.  

The second part of this proposition follows directly from Proposition~\ref{Int} and the formula \eqref{CC2} for the boundary of the critical cone. 
\end{proof}

\section{Numerical Experiments}\label{NE}
\setcounter{equation}{0}

In this section, we conduct numerical experiments to illustrate the solution uniqueness for the nuclear norm minimization problem \eqref{NNM}. These experiments are performed for different matrix ranks $r$ and numbers of measurements $m$. 


In the experiment, we generate $X_0$, an $n\times n$ matrix of rank $r$, by sampling two factors $W\in\mathbb R^{n\times r}$ and $H\in\mathbb R^{n\times r}$ with independent and identically distributed (i.i.d.) random entries and setting $X_0 = WH^*$. We vectorize problem \eqref{NNM} in the following form:
\begin{equation}\label{vecNNM}
    \min_{X\in \R^{n\times n}}\quad \|X\|_*\quad \mbox{subject to}\quad \Phi\ \text{vec}(X) = \Phi\ \text{vec}(X_0),
\end{equation}
where $\Phi $ is an ${m\times n^2}$ matrix  drawn from the standard Gaussian ensemble. The point $X_0$ is said to be {\em recovered} if solving problem \eqref{vecNNM} gives us an optimal solution $\OX$ that is relatively close to $X_0$, in particularly, $\|\OX - X_0\|_F/\|X_0\|_F < 10^{-3}$ as proposed in \cite{CR09}. We use cvxpy package to solve this problem to find $\OX$.

In order to check whether $X_0$ is a unique solution of problem \eqref{vecNNM}, by Corollary~\ref{UniRaNu} and Remark~\ref{Max}, we have to verify that  the following nonconvex optimization problem has global maximum value $0$: 
\begin{equation}\label{ncvx}
\begin{aligned}
&\max\; \frac{1}{2}\|A\|^2_F+\frac{1}{2}\|C\|_F^2\; \;\\
&\mbox{subject to }\; \;\Phi\left(U\begin{pmatrix}A&BC&0\\C^TB^T&C&0\\0&0&0\end{pmatrix}V^T\right)=0, A\in \mathbb{S}^r,B\in \R^{r\times (p-r)}, C\in \mathbb{S}_+^{p-r},
 \end{aligned}
\end{equation}
where $(U,V)\in \mathcal{O}(X_0)\cap\mathcal{O}(Y)$ is from~\eqref{SSVD} with $Y\in \Delta(X_0)$ being a particular dual certificate computed as in Remark \ref{ComY}, $r={\rm rank}\, (X_0)$, and $p=p(Y)$ from \eqref{p}. Indeed, let $\Hat U\Sigma \Hat V^T$ be an SVD of $\Phi$,  $\Hat V_G$ be the matrix whose columns are the last $n^2-r$ columns of $V$, and $U_0\begin{pmatrix}\Sigma_0&0\\0&0\end{pmatrix} V_0^T$ be a full SVD of $X_0$. To find a particular certificate of \eqref{vecNNM}, we solve the following problem by using the \texttt{cvxpy} package:
\begin{equation*}
 \min_{Z\in\TT_0^\perp}\|Z\|\quad \mbox{subject to}\quad N\text{vec}(Z) = - N\text{vec}(E_0),
\end{equation*}
where $N = V_G^T$, $E_0 = U_0\begin{pmatrix}\Id_r&0\\0&0\end{pmatrix} V_0^T$, and $\TT_0^\perp$ is known by
\begin{equation*}
    \TT_0^\perp=\left\{U_0\begin{pmatrix}0&0\\0&D\end{pmatrix}V_0^T|\; D\in \R^{(n-r)\times (n - r)}\right\}.
\end{equation*}
Let $\rho(X_0)$ be the optimal value of this problem and $Z_0$ be an optimal solution. Note also from \cite[Proposition~4.4 and Remark~4.5]{FNT21}, $X_0$ is an optimal solution of problem~\eqref{vecNNM} if and only if $\rho (X_0)\le 1$. Due to possible errors in numerical computation, we increase this bound to $\rho(X_0)<1.001$ for optimal solutions.  It is also known from \cite[Section~4]{FNT21} that $Y=Z_0+E_0$ is a particular dual certificate of \eqref{vecNNM}.

Back to the maximum problem above, we rewrite it  as follows
\begin{equation}\label{rnoncvx}
    \begin{aligned}
        & \max\;\frac{1}{2}\|A\|^2_F+\frac{1}{2}\|C\|_F^2\\
        & \mbox{s.t.}\; \;\Phi\left(U\begin{pmatrix}A&D&0\\ D^T&C&0\\0&0&0\end{pmatrix}V^T\right)=0, D = E, E = BC, A\in \mathbb{S}^r,B\in \R^{r\times (p-r)}, C\in \mathbb{S}_+^{p-r}.
    \end{aligned}
\end{equation}
It is worth noting that this problem takes the format of the one considered in \cite{BST18}, for which a multiblock version of ALBUM (Adaptive Lagrangian-Based mUltiplier Method) \cite{BST18}, named mALBUM, is investigated. The algorithm mALBUM for solving \eqref{rnoncvx} iteratively generates a sequence $\{A^k,C^k,B^k,D^k,E^k,F^k, G^k)\}$ by
\begin{equation*}
    \begin{aligned}
        (A^{k+1},C^{k+1}, D^{k+1}) \in &\arg\min - \langle A^k, A\rangle - \langle C^k, C\rangle - \langle F^k, D\rangle  - \langle G^k,B^kC\rangle + \frac{\beta_k}{2}\|E^k-D\|^2\\
        &\quad\quad + \frac{\beta_k}{2}\|E^k-B^kC\|^2\\
        &\mbox{s.t.}\; \;\Phi\left(U\begin{pmatrix}A&D&0\\ D^T&C&0\\0&0&0\end{pmatrix}V^T\right)=0, A\in \mathbb{S}^r, C\in \mathbb{S}_+^{p-r},\\
        B^{k+1} \in &\arg\min  - \langle G^k,BC^{k+1}\rangle  + \frac{\beta_k}{2}\|E^k-BC^{k+1}\|^2,\\
        E^{k+1} = &  \left(-F^k - G^k\right)/(2\beta_k) +(D^{k+1}+B^{k+1}C^{k+1})/2,\\
        F^{k+1} = & F^k + \beta_k(E^{k+1}-D^{k+1}),\\
        G^{k+1} = & G^k + \beta_k(E^{k+1}-B^{k+1}C^{k+1}).
    \end{aligned}
\end{equation*}
Here, the parameter $\beta$ is updated adaptively by using the following rule: 
if \[
\|E^{k+1}-D^{k+1}\| + \|E^{k+1}-B^{k+1}C^{k+1}\| > M\quad \mbox{ for some large constant $M$},\]
then $\beta_{k+1} \leftarrow \nu\beta_k$ with some constant $\nu>0$; otherwise, $\beta_{k+1} \leftarrow \beta_k$. In our experiment, we set $\nu = 1.2$, $M = 10^5$, and $\beta_0 = 1$. We use the \texttt{cvxpy} package to solve sub-problems. 

To start the algorithm, the initial values of  $(A^0,C^0,B^0,D^0,E^0,F^0,G^0)$ are chosen as follows: $A^0 = (V_1 + V_1^T)$, where $V_1\in \mathbb R^{r\times r}$ is drawn from the standard Gaussian ensemble, $B^0$ is drawn from the standard Gaussian ensemble, $C^0 = V_2V_2^T$, where $V_2\in \mathbb R^{(p-r)\times (p-r)}$ is drawn from the standard Gaussian ensemble, $D^0 = E^0 = B^0C^0, F^0 = 0$, and $G^0 = 0$.

We classify solution uniqueness at $X_0$ if $\|\overline A\| + \|\overline C\| < 10^{-5}$ and $\rho(X_0) < 1.001$, where $\overline A$ and $\overline C$ are outputs of mALBUM. 
To demonstrate the occurrence of solution uniqueness for problem \eqref{vecNNM}, we create graphs that show the proportion of solution uniqueness with respect to the number of measurements in Figures \ref{fig:NNM}. For each fixed value of $n=40$ and $2\le r\le 7$, we conduct a simulation study on 100 random instances at each number of measurements $m$. At each $m$, the proportion of cases (among $100$) of $X_0$ that are recovered by solving problem \eqref{vecNNM} are depicted as the red curve \cite{CR09}. The percentage of cases of $X_0$ that are unique solutions of problem~\eqref{vecNNM} based on the above classification is displayed as the blue curve. It is natural to see that all solutions $X_0$ that can be recovered by solving problem \eqref{vecNNM} using the \texttt{cvxpy} package are unique solutions. But the most important observation from Figure~\ref{fig:NNM} is the minimal difference between the red and blue curves.  This suggests that our method in this paper sharply predicts solution uniqueness for nuclear norm minimization problems \eqref{NNM}. Actually, the curve blue is slightly above the red one. The maximum gap between them is around 3-4\% at different ranks, but for most of $m$ (number of observations) it is unnoticeable. In our opinion this small difference occurs  because the method mALBUM above solves problem \eqref{rnoncvx} with only {\em critical solutions}, which may be not global solutions (see discussion in Remark~\ref{Max}). It is possible that some of the solutions of problem \eqref{ncvx} are still in the form $(0,B,0)\in \mathbb{S}^r\times \R^{r\times(p-r)}\times \mathbb{S}^{p-r}_+$, but the global optimal value is not $0$. In this case the corresponding $X_0$ is not a unique solution, but this situation seems to be rare from our experiments described in Figure~\ref{fig:NNM}.

\begin{figure*}[!htpb]
\vspace{-1ex}
\begin{center}
\includegraphics[width=1\linewidth]{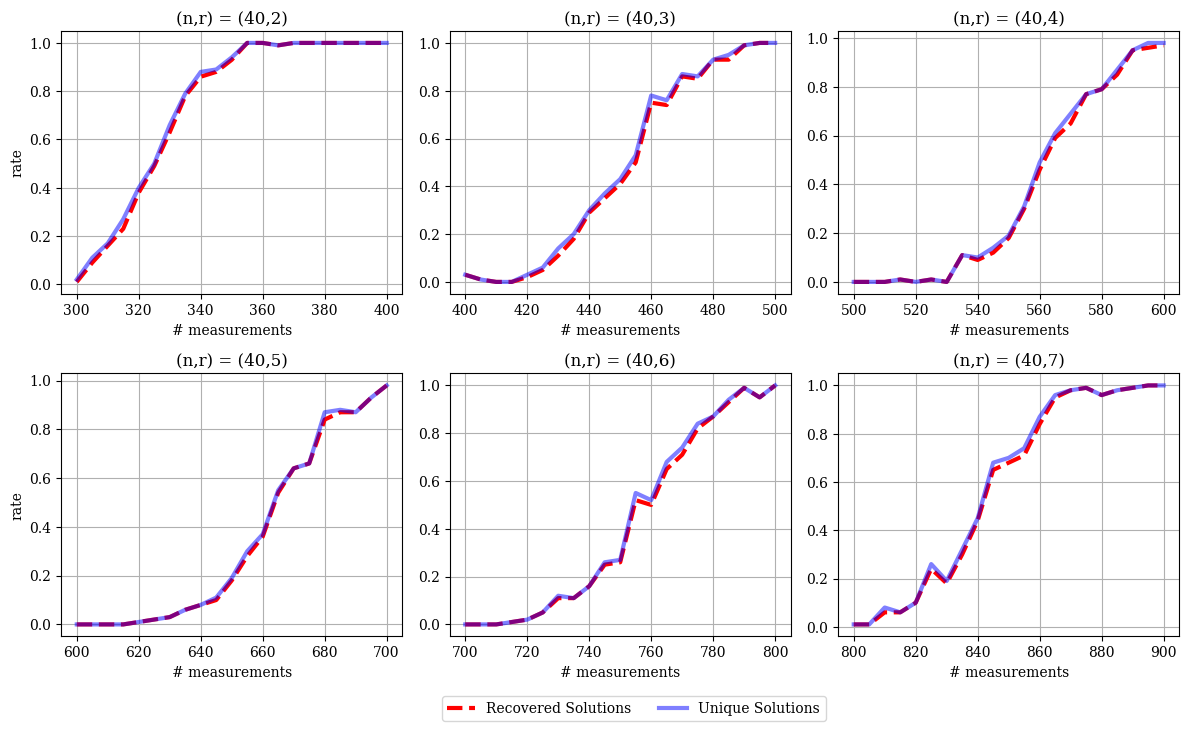}
\caption{Proportions of cases for which $X_0$ is the unique solution with respect to the number of measurements.\label{fig:NNM}} 
\end{center} 
\vspace{-3ex}
\end{figure*}


\begin{thebibliography}{99}

\bibitem{ALMT14} D. Amelunxen, M. Lotz, M. B. McCoy, and J. A. Tropp: Living on the edge: Phase transitions in convex programs with random data, Information and Inference: A Journal of the IMA, \textbf{3} (2014), 224--294.

\bibitem{AG08} F. J. Arag\'on Artacho and M. H.  Geoffroy: Characterizations of metric regularity of subdifferentials, \textit{J. Convex Anal.}, {\bf 15} (2008), 365--380.


\bibitem{BDE09} A. M. Bruckstein, D. L. Donoho, and M. Elad. From sparse solutions of systems of equations
to sparse modeling of signals and images, {\em SIAM Rev.}, {\bf 51} (2009), 34--81.

\bibitem{BLN22} Y. Bello-Cruz, G. Li, and T. T. A. Nghia: Quadratic growth condition and uniqueness of optimal solution to Lasso, \textit{J. Optim. Theory Appl.}, {\bf 194} (2022), 167--190.

 \bibitem{BNPS17} J. Bolte, T. P. Nguyen, J. Peypouquet, B. W. Suter: From error bounds to the complexity of first-order descent methods for convex functions, \textit{Math. Program.}, {\bf 165} (2017), 471--507. 

\bibitem{BS00}
J.~F. Bonnans and A.~Shapiro.
\newblock {\em Perturbation analysis of optimization problems}.
\newblock Springer Series in Operations Research. Springer-Verlag, New York,
  2000.

  \bibitem{BST18}
  J. Bolte, S. Sabach, and  M. Teboulle: Nonconvex Lagrangian-based optimization: monitoring schemes and global convergence, \textit{Mathematics of Operations Research}, {\bf 43} (2018), 1210--1232.

  

  \bibitem{CDZ17} Y. Cui, C. Ding, and X.  Zhao: Quadratic growth conditions for convex matrix optimization problems associated with spectral functions, \textit{SIAM J. Optim.}, {\bf 27} (2017), 2332--2355.

  \bibitem{CRPW12} V. Chandrasekaran, B. Recht, P.A Parrilo, and A. S. Willsky: The convex geometry of linear inverse problems, \textit{Found Comput Math}, \textbf{12} (2012), 805--849.


  \bibitem{CP10} E.  Cand\`es
and Y. Plan: Matrix completion with noise, \textit{Proceeding of the IEEE}, \textbf{98} (2010), 925--936.

\bibitem{CR09} E. Cand\`es and B. Recht: Exact matrix completion via convex optimization \textit{Found. Comput. Math.}, \textbf{9} (2009), 717--772.

\bibitem{CR13}  E. Cand\`es and B. Recht: Simple bounds for recovering low-complexity models, \textit{Math. Program.},  \textbf{141} (2013), 577--589.

\bibitem{CT05} E. J. Cand\`es and T. Tao. Decoding by linear programming, \textit{IEEE Trans. Inform. Theory},
{\bf 51} (2005), 4203--4215.

\bibitem{C78} L. Crome: Strong uniqueness, \textit{Numer. Math.}, \textbf{29} (1978), 179--193.

\bibitem{C91} I. Csisz\'ar: Why least squares and maximum entropy? An axiomatic approach to inference for linear
inverse problems, \textit{Ann. Statist.}, {\bf 19} (1991), 2032--2066.

\bibitem{DMN14} D. Drusvyatskiy, B. S. Mordukhovich, and T. T. A.  Nghia: Second order growth, tilt stability, and metric regularity of
the subdifferential,  \textit{J. Convex Anal.}, {\bf 21} (2014), 1165--1192.

\bibitem{DH01} D. L. Donoho and X. Huo: Uncertainty principles and ideal atomic decomposition, \textit{IEEE
Trans. Inform. Theory}, {\bf 47} (2001), 2845--2862.

\bibitem{FR13} S. Foucart and H. Rauhut: \textit{A Mathematical Introduction to Compressive Sensing}, Birkh\"auser, New York, 2013.

\bibitem{F05} J. J. Fuchs: Recovery of exact sparse representations in the presence of bounded noise, \textit{IEEE Trans. Inf. Theory}, \textbf{51} (2005), 3601--3608.

\bibitem{FNT21} J. Fadili, T. T. A. Nghia, and T. T. T. Tran: Sharp, strong and unique minimizers for low complexity robust recovery, \textit{IMA Inf. Inference}, \textbf{12} (2023), 1461--1513.

\bibitem{FNP23} J. Fadili, T. T. A. Nghia, and D. N. Phan: Geometric characterizations for strong minima with applications to nuclear norm minimization problems, arXiv:2308.09224v1 (2023).

\bibitem{G17} J. C. Gilbert: On the solution uniqueness characterization in the $\ell_1$ norm and polyhedral gauge recovery, \textit{ J. Optim. Theory Appl.},  \textbf{172} (2017), 70--101.

\bibitem{G11} M. Grasmair. Linear convergence rates for Tikhonov regularization with positively homogeneous
functionals, \textit{Inverse Problems},  {\bf 27} (2011), 075014, 16.

\bibitem{GHS11}  M. Grasmair, M.  Haltmeier, O. Scherzer: Necessary and sufficient conditions for linear convergence of $\ell_1$-regularization,  \textit{Comm.  Pure Applied Math.}, \textbf{64} (2011), 161--182.

\bibitem{HJ85} R. A. Horn and C. R. Johnson: {\em Matrix Analysis}, Cambridge University Press, New York, 1985. 

\bibitem{HP23} T. Hoheisel and E. Paquette: Uniqueness in nuclear norm minimization: Flatness of the nuclear norm sphere and simultaneous polarization, \textit{J. Optim. Theo. Appl.}, {\bf 197} (2023), 252?-276.

\bibitem{HKS23} J. He, C. Kan, and W. Song: On solution uniqueness and robust recovery for sparse regularization
with a gauge: from dual point of view,  arXiv:2312.11168 (2023).


\bibitem{HS23} N. T. V. Hang and E. Sarabi: Smoothness of subgradient mappings and its applications in parametric optimization, arXiv:2311.06026 (2023).

\bibitem{LS05a} A. S. Lewis and H. S. Sendov: Nonsmooth analysis of singular values. I. Theory.
\textit{Set-Valued Anal.}, {\bf 13} (2005), 213--241.

\bibitem{LS05b} A. S. Lewis and H. S. Sendov: Nonsmooth analysis of singular values. II. Applications, \textit{Set-Valued Anal.}, {\bf 13} (2005), 243--264.

 \bibitem{LZ13} A. S. Lewis and S. Zhang: Partial smoothness, tilt stability, and generalized Hessians, {\em SIAM J. Optim.}, {\bf 23} (2013), 74--94. 

\bibitem{LPB21} Y. Liu, S. Pan, and S. Bi: Isolated calmness of solution mappings and exact recovery conditions for nuclear norm optimization problems, {\em Optimization}, {\bf 70} (2021), 481--510.

\bibitem{JKL15} J. S. J{\o}rgensen, C. Kruschel, and D. A. Lorenz. Testable uniqueness conditions for empirical
assessment of undersampling levels in total variation-regularized X-ray CT, {\em  Inverse Probl. Sci.
Eng.}, {\bf 23} (2015), 1283--1305.



 \bibitem{M92} B. S. Mordukhovich: Sensitivity analysis in nonsmooth optimization, in {\em Theoretical Aspects of Industrial Design} (D. A. Field and V. Komkov, eds.), SIAM Proc. Applied Math., Vol. 58, pp. 32--46, SIAM, Philadelphia, PA, 1992.

 \bibitem{M1} B. S. Mordukhovich: {\em Variational Analysis and Generalized Differentiation}, Springer, 2006.

 \bibitem{MS19} S. Mousavi and J. Shen: Solution uniqueness of convex piecewise affine functions based optimization with applications to constrained $\ell_1$-minimization, \textit{ESAIM: Control Optim. Cal. Variations}, \textbf{25} (2019), No. 56, 27. 


 \bibitem{MR12} B. S. Mordukhovich and R. T. Rockafellar:
\newblock Second-order subdifferential calculus with applications to tilt stability in optimization.
\newblock {\em SIAM J. Optim.},
22:953--986, 2012.

 \bibitem{P79} B. T. Polyak: Sharp minima. Technical report, Institute of Control Sciences Lecture Notes,
Moscow, USSR, 1979.

\bibitem{P87} B. T. Polyak: {\em Introduction to Optimization}, Optimization Software, Inc., Publications Division, New York, 1987.

  \bibitem{PR98}
R.~A. Poliquin and R.~T. Rockafellar: Tilt stability of a local minimum.
\newblock {\em SIAM J. Optim.}, 8(2):
287--299, 1998.


\bibitem{RRN12} N. Rao, B. Recht, and R. Nowak: Universal measurement bounds for structured sparse signal
recovery {\em In Proceedings of AISTATS} (2012).

\bibitem{RF08} V. Roth and B. Fischer: The group-lasso for generalized linear models: uniqueness of solutions
and efficient algorithms, In {\em Proceedings of the 25th international conference on Machine
learning} (2008).

\bibitem{R70} R. T. Rockafellar: {\em Convex Analysis}, Princeton University Press, Princeton, New Jersey, 1970.

\bibitem{RW98} R. T. Rockafellar and  R. J-B. Wets: {\em Variational Analysis}, Springer, Berlin, 1998.

\bibitem{T13} R. J. Tibshirani: The Lasso problem and uniqueness, {\em Electron. J. Stat.}, {\bf 7} (2013), 1456--1490.

\bibitem{VPDF13} S. Vaiter, G. Peyr\'e, C. Dossal, and J. Fadili. Robust sparse analysis regularization,  \textit{IEEE Trans.
Inform. Theory}, {\bf 59} (2013), 2001--2016.

\bibitem{W92} G. A. Watson: Characterization of the subdifferential of some matrix norms, \textit{Linear Algebra Appl.}, {\bf 170} (1992), 33--45.

 

\bibitem{ZS17} Z. Zhou and A. M-C.  So: A unified approach to error bounds for structured convex optimization, \textit{Math. Program.}, {\bf 165} (2017), 
689--728.

\bibitem{ZT95}
R. Zhang and J. Treiman:
 Upper-Lipschitz multifunctions and inverse subdifferentials, \textit{\em Nonlinear Anal.}, {\bf 24} (1995), 273--286.

\bibitem{ZYC15} H. Zhang, W. Yin, and  L. Cheng: Necessary and sufficient conditions of solution uniqueness in 1-norm minimization,  \textit{ J. Optim. Theory Appl.}, \textbf{164} (2015), 109--122

\bibitem{ZYY16}  H. Zhang, H.,  M. Yan, and W.  Yin: One condition for solution uniqueness and robustness of both $\ell_1$-synthesis and $\ell_1$-analysis minimization, \textit{Adv. Comput. Math.}, \textbf{42} (2016), 1381-1399.

\end{thebibliography}
\end{document}